\documentclass[12pt,reqno]{amsart}
\usepackage{amsmath}
\usepackage{amssymb}
\usepackage{fullpage}
\usepackage{algorithm}
\usepackage{algpseudocode}
\usepackage{graphics}
\usepackage[colorlinks,citecolor=blue]{hyperref}
\usepackage{booktabs}
\usepackage{cite}
\usepackage{enumitem}
\usepackage{tikz,pgfplots}
\usetikzlibrary{shapes.multipart,backgrounds,shapes,fit}
\usetikzlibrary{arrows,decorations.markings,shapes.arrows,arrows.meta}
\usepackage{stix}
\usepackage{subcaption}

\newcommand{\R}{\mathbb{R}}
\renewcommand{\SS}{\mathbb{S}}

\DeclareMathOperator{\aut}{\operatorname{Aut}}

\DeclareMathOperator{\bd}{\operatorname{bd}}
\DeclareMathOperator{\cl}{\operatorname{cl}}

\DeclareMathOperator{\Tr}{\operatorname{Tr}}

\DeclareMathOperator{\inte}{\operatorname{int}}

\DeclareMathOperator{\dom}{\operatorname{dom}}

\newcommand{\cT}{\mathcal{T}}
\newcommand{\Ke}{{K}_\text{exp}}

\newcommand{\norm}[1]{\left|\left|#1\right|\right|}

\newtheorem{thm}{Theorem}[section]
\newtheorem{lem}[thm]{Lemma}
\newtheorem{cor}[thm]{Corollary}
\newtheorem{dfn}[thm]{Definition}
\newtheorem{prop}[thm]{Proposition}

\theoremstyle{remark}
\newtheorem{examp}[thm]{Example} 

\newcommand{\iprod}[2]{\left\langle {#1}, {#2} \right\rangle}

\newcounter{algorithmctr}[section]
\renewcommand{\thealgorithmctr}{\thesection.\arabic{algorithmctr}}
    {\refstepcounter{algorithmctr}\begin{list}{}{%
        \setlength{\rightmargin}{.05\linewidth}%
        \setlength{\leftmargin}{.05\linewidth}}%
        \item[]{\setlength{\parskip}{0ex}\bigskip\par%
         \nopagebreak%
         \underline{{\bf Algorithm \thealgorithmctr.} \emph{#1.}}}}%
    {{\setlength{\parskip}{-1ex}\nopagebreak\smallskip\par} \end{list}}

\begin{document}

\title[Complexity of primal--dual ipms]
{New complexity bounds for primal--dual interior-point algorithms in conic optimization}

\author{Joachim Dahl \and Levent Tun\c{c}el \and Lieven Vandenberghe}
\thanks{Joachim Dahl: Cardinal Operations, Ltd., China (e-mail: dahl.joachim@gmail.com).\\
Levent Tun\c{c}el: Department of Combinatorics and Optimization, 
Faculty of Mathematics, University of Waterloo, Waterloo, Ontario N2L 3G1,
Canada (e-mail: levent.tuncel@uwaterloo.ca).
Research of this author was supported in part by Discovery Grants from 
NSERC and by U.S.\ Office of Naval Research under award number
N00014-18-1-2078.  \\
Lieven Vandenberghe: Department of Electrical and Computer
Engineering, UCLA, Los Angeles, CA 90095, USA (e-mail: vandenbe@ucla.edu). 
}

\date{September 9, 2025, revised: \today}

\begin{abstract}
We provide improved complexity results for symmetric primal--dual 
interior-point algorithms in conic optimization.
The results follow from new uniform bounds on a key complexity measure 
for primal--dual metrics at pairs of primal and dual points.
The complexity measure is defined as the largest eigenvalue
of the product of the Hessians of the primal and dual barrier functions,
normalized by the proximity of the points to the central path. 
For algorithms based on self-scaled barriers for symmetric cones, 
we determine the exact value of the complexity measure. 
In the significantly more general case of self-concordant barriers
with negative curvature, we provide the asymptotically tight upper bound 
of 4/3.  This result implies $O(\vartheta^{1/2}\ln(1/\epsilon))$ iteration 
complexity for a variety of symmetric (and some nonsymmetric)
primal--dual interior-point algorithms. 
Finally, in the case of general self-concordant barriers, we give 
improved bounds for some variants of the complexity measure.
\end{abstract}

\maketitle

\begin{section}{Introduction}
\label{sec:intro}

Primal--dual algorithms are among the most efficient, both in theory and in 
practice, and most robust algorithms for convex optimization.  
Primal--dual scalings (metrics) play a key role in the primal--dual 
algorithms, including interior-point algorithms. 
In this paper, we provide new bounds on complexity measures
for primal--dual scalings. 
The new bounds directly yield iteration complexity upper bounds
for a class of primal--dual interior-point algorithms.

Let $K \subset \R^n$ be a pointed, closed, and convex cone with nonempty 
interior. We call such cones \emph{regular}. If $K$ is a regular cone, then so is the dual cone
\begin{equation} \label{e-dual-cone}
K^* := \left\{ s \in \R^n: \iprod{s}{x} \geq 0 \;
 \textup{ for all } x \in K\right\}.
\end{equation}
Given a linear transformation $A: \R^n \to \R^m$, $b \in \R^m$, 
$c \in \R^n$, we define a conic optimization problem as
\renewcommand{\theequation}{P}%
\begin{equation}
\addtocounter{equation}{-1}
\label{e-primal}
\begin{array}{ll}
\mbox{minimize} & \iprod{c}{x} \\ 
\mbox{subject to} & A(x) =b \\ &  x \in K,
\end{array}
\end{equation}
with variable $x\in \R^n$.  The dual problem is defined as
\renewcommand{\theequation}{D}%
\begin{equation}
\addtocounter{equation}{-1}
\label{e-dual}
\begin{array}{ll}
\mbox{maximize} & b^\top y \\
\mbox{subject to} & A^*(y)+s = c \\ & s \in K^*,
\end{array}
\end{equation}
\renewcommand{\theequation}{\arabic{equation}}%
with variables $y\in\R^m$, $s\in\R^n$, where the linear transformation 
$A^* :\R^m \to \R^n$ (the \emph{adjoint} of $A$) is defined by the identity
\[
\iprod{A^*(y)}{x} = y^{\top} A(x), \,\,\,\, \forall x \in \R^n, \,\, 
 \forall y \in \R^m.
\]
Algorithms which simultaneously solve~\eqref{e-primal} 
and~\eqref{e-dual} by generating a sequence of primal--dual 
iterates $x^{(k)} \in \inte(K)$ and $s^{(k)} \in \inte(K^*)$ are relevant 
to this paper.  Among such algorithms, we focus on those which can be
interpreted as using a variable primal--dual scaling in each iteration. 
Specifically, a scaling is a symmetric positive definite matrix $T^{(k)}$
used in the equations 
\[
A(d_x) = 0, \qquad A^*(d_y) + d_s = 0, \qquad
T^{(k)}d_x +  d_s = r, 
\]
which determine the primal and dual search directions $d_x,d_y,d_s$
at iteration $k$.  
Primal--dual scalings are also known as primal--dual metrics because of 
their close connections with variable metric methods in the optimization
literature including the usage of variable metrics in quasi-Newton methods.

The scaling used in primal--dual interior-point methods for linear 
programming (with $K=\R^n_+$) is the diagonal matrix with diagonal entries 
$T^{(k)}_{ii} = s_i^{(k)}/x_i^{(k)}$; see \cite{Wri:97}.
For semidefinite and second-order cone programming, a wide variety
of scalings which generalize the primal--dual scaling for $K=\R^n_+$
have been studied and used successfully (see, for instance~\cite{Todd2001} and the references therein). 
The most complete theory for primal--dual scalings
is the theory of \emph{self-scaled barriers} due to Nesterov and 
Todd~\cite{NT1997,NT1998}, which covers the cases when the cone $K$ is 
the nonnegative orthant, $\R^n_+$, the cone of symmetric positive 
semidefinite matrices over reals, $\SS^n_+$, the cone of Hermitian 
positive semidefinite matrices over complex numbers, $\mathbb{H}^n_+$,
and the direct sum (or cartesian product) of second order cones. 
More precisely, the theory of self-scaled barriers treats the cases when $K$
is a \emph{symmetric} (i.e., \emph{homogeneous} and \emph{self-dual}) cone
\cite{ScA:01,ScA:03,AlG:03}.
The primal--dual scaling resulting from the theory of self-scaled barriers is
known as  the \emph{Nesterov--Todd scaling}, and is the scaling used 
in several popular software packages, including SeDuMi~\cite{Sturm1999}, 
SDPT3~\cite{SDPT3-2012}, MOSEK~\cite{ART:03}, and 
Clarabel~\cite{GoulartChen2024}. 

Since Nesterov and Todd's ground-breaking work on self-scaled barriers 
for symmetric cones, there has been substantial effort to extend some 
aspects of their primal--dual algorithms to nonsymmetric 
cones~\cite{Tuncel2001,NemirovskiT2005,Nesterov2012,MT2014,SkaajaYe2015,KarimiT2020,DahlA2022,BD2022,TV2023,ChG:25,ADV:10,PaY:22,Ser:15,GoulartChen2024}.
While not all nice properties of self-scaled barriers can be extended 
beyond symmetric cones (see \cite[Theorem 4.2]{Tuncel1998}, 
\cite[Lemma 6.4]{NT2016}), some of the critical, desired properties
of primal--dual scalings induced by self-scaled barriers were extended to 
all logarithmically homogeneous self-concordant barriers 
(LHSCBs) and all convex cones~\cite{Tuncel2001}. 
In \cite{Tuncel2001}, a complexity measure of these primal--dual scalings is
defined, which yields upper bounds on the iteration complexity of a class
of interior-point algorithms utilizing the primal--dual scalings. 
In this paper, this complexity measure (defined below in 
Section~\ref{subsec:1.1}) will be denoted by $\xi_F$, where $F$ is a 
$\vartheta$-logarithmically homogeneous self-concordant barrier 
($\vartheta$-LHSCB) for a convex cone $K$~\cite{NN1994}.  The complexity
measure $\xi_F$ is based on generalized eigenvalues of primal and dual
Hessians relative to proximity of the arguments of the Hessians to
the central path. The complexity measure
$\xi_F$ also admits a geometric interpretation
in terms of Dikin ellipsoids (defined in Section~\ref{subsec:LHSCB}). Locally, for
a fixed pair of primal--dual interior-points $(x,s)$, $\xi_F$ is related 
to the smallest blow-up factor for the Dikin ellipsoid at $x$ so that
it contains the dual of the Dikin ellipsoid at $s$ when centered at $x$ (also
see the detailed explanation in Section~\ref{s-neg-curv-xi} after equation
\eqref{e-hatcheck-T}).
The central role of $\xi_F$ in the complexity analysis 
from~\cite{Tuncel2001} can be summarized as follows.
\begin{itemize}
\item  In~\cite[Section 5]{Tuncel2001}, primal--dual interior-point
algorithms for conic linear programming are described that achieve an
\begin{equation} \label{e-xi-iter-bound}
O(1) \xi_F \vartheta^{1/2}\ln(1/\epsilon)
\end{equation}
iteration bound\footnote{Here, $\epsilon \in (0,1)$, is a user-defined 
desired accuracy; the algorithms are required to generate a pair of 
primal--dual feasible solutions whose duality gap is at most $\epsilon$ 
times the original duality gap of the starting feasible primal--dual pair.} 
where $O(1)$ denotes an absolute constant and $F$
is the barrier used in the algorithm. 
\item
It is known that $\xi_F \leq 9\vartheta(\vartheta-1)$ 
for any $\vartheta$-LHSCB $F$ \cite{Obro2019}.
Combined with the iteration complexity bound~(\ref{e-xi-iter-bound}), 
this gives an iteration bound
\[
O(1) \vartheta^{5/2} \ln(1/\epsilon).
\]

\item If we consider a class of conic problems~\eqref{e-primal} with
different cones $K$, each with a corresponding barrier $F$,
and $\xi_F$ is bounded by an absolute constant for all these barriers
(e.g., $\xi_F \leq 10$ for every such barrier),
then the iteration complexity for solving any instance in that class of
conic problems is 
\[
 O\left(\vartheta^{1/2}\ln(1/\epsilon)\right).
\]
In particular, when $\xi_F \leq 4/3$, 
then (using the analysis in~\cite[Section 5]{Tuncel2001})
$20\vartheta^{1/2}\ln(1/\epsilon)$ iterations suffice
(the constant 20 can be improved). 
\end{itemize}

The paper \cite{Tuncel2001} also describes the properties  
of primal--dual scalings that are needed for these complexity bounds to
apply.
At each iteration of these algorithms, a primal--dual scaling can
be selected from a convex set $\cT(x^{(k)},s^{(k)}; \hat \xi)$ where 
$\hat\xi\geq \xi_F$ is a constant algorithm parameter. 
The set $\cT(x^{(k)},s^{(k)};\hat \xi)$ of acceptable scalings is defined 
by two systems of linear equality constraints and two matrix inequalities 
(see definition~\eqref{e-T}).

This paper contains several new results on the complexity measure $\xi_F$
and the associated sets of primal--dual scalings $\cT(x,s;\hat\xi)$.
\begin{itemize}
\item It is known that for self-scaled barriers of symmetric cones,
$\xi_F  \leq 4/3$, and that the Nesterov--Todd scaling is an element
of $\cT(x,s;4/3)$ \cite[Theorem 6.1]{Tuncel2001}.
In Section~\ref{sec:negative_curvature} we study $\xi_F$
for LHSCBs with negative curvature, i.e., LHSCB barriers $F$ with the 
additional property that $\langle F'(x), y\rangle$ is a concave function 
of $x$ for any fixed $y\in K$.  
The class of convex cones which admit LHSCBs with negative curvature
is much larger than the symmetric cones.  For example,
any barrier 
$F(x) := G(B(x))$ for a regular cone $K := \{x \, : \, B(x) \in C\}$, where 
$B$ is a suitable linear function, $C$ is a symmetric cone, 
and $G$ a self-scaled barrier for $C$, is an LHSCB with negative curvature,
but not necessarily self-scaled.
We show that the bound $\xi_F \leq 4/3$ holds for the family
of LHSCBs with negative curvature.
We further prove that the primal--dual integral scaling from \cite{MT2014} is 
an element of $\cT(x,s;4/3)$.  An immediate consequence of this result 
is that the primal--dual potential reduction algorithm 
from~\cite{Tuncel2001},
when implemented using the integral primal--dual scalings~\cite{MT2014}, has
iteration complexity matching the current best bounds for conic 
optimization problems using cones with negative curvature barriers.
This result impacts other primal--dual algorithms, e.g., the 
predictor--corrector algorithm analyzed in~\cite{MT2014}. 
For negative curvature LHSCBs, our results imply that in some algorithms,
during some iterations, we may take 
significantly larger predictor steps and still maintain 
$O(\vartheta^{1/2}\ln(1/\epsilon))$ iteration complexity
bounds.

\item
In Section~\ref{sec:symmetric} we show that for optimal 
$\vartheta$-self-scaled barriers $F$ with $\vartheta\geq 2$,
\[
 \xi_F = \frac{(\tau_\vartheta +1)^2}{\tau_\vartheta (\tau_\vartheta +2)}, 
\quad \mbox{where\ } \tau_\vartheta := \sqrt{\frac{\vartheta}{\vartheta-1},}
\]
and that both the Nesterov--Todd scaling and the integral scaling 
from \cite{MT2014} are in $\cT(x,s;\xi_F)$. 
The limit of $\xi_F$ for $\vartheta \rightarrow +\infty$
is $4/3$, so the $4/3$ upper bound is asymptotically tight
(see Figure~\ref{f-xi-hat-check}).   

\item In the most general setting, with $K$ any regular cone
and $F$ any $\vartheta$-LHSCB, {\O}bro~\cite{Obro2019} has proved that 
$\xi_F \leq 9\vartheta(\vartheta-1)$ (assuming $\vartheta \geq 2$).  
In Section~\ref{sec:general_cones}, utilizing some results from~\cite{NT2016},
we obtain new upper bounds, for a relaxation $\widecheck{\xi}_F$ of $\xi_F$ (defined in
\eqref{s-neg-curv-xi}).
We also include numerical experiments which provide insights into the conditions 
which characterize good primal--dual scalings, and properties of pairs of 
primal and dual interior points where finding a good scaling is the most
difficult.
\end{itemize}

Section~\ref{sec:conclusion} contains some concluding remarks and future 
research directions.
\end{section}

\begin{section}{Self-concordant barriers and primal--dual scalings}
Throughout the paper we assume $K$ is a \emph{regular} cone 
in $\R^n$ and $K^*$ is the dual cone defined in~\eqref{e-dual-cone}.
We occasionally  use the generalized inequality notation 
$x\succeq_K 0$ and $u\succeq_{K^*} 0$ as an alternative for
$x\in K$ and $u \in K^*$.
Without subscript, the inequality $A \preceq B$, where $A$ and $B$ are
symmetric matrices, means that $B-A$ is positive semidefinite.

In this section we review some properties of logarithmically
homogeneous self-concordant barriers, and the basic definitions
and results on primal--dual scalings from \cite{Tuncel2001,MT2014}.

\begin{subsection}{Minkowski gauge and norm}
The \emph{Minkowski gauge} for the cone $K$ is the function
that assigns to $x\in\inte(K)$ and $h\in\R^n$ the value
\begin{equation}
\sigma_x(h) := \inf{\{ \beta \geq 0 : \beta x -h \in K\}}.
 \label{e-mink-def} 
\end{equation}
The function $\sigma_x(\cdot)$ is convex, nonnegative, and positively
homogeneous.
If $-h\in K$, then $\sigma_x(h)=0$.  
If $-h\not\in K$, then $\sigma_x(h)> 0$ and $1/\sigma_x(h)$ is the step size
from $x$ to the boundary of $K$ in the direction~$-h$.
We also note the identity
\begin{equation} \label{e-mink-translate}
1+\sigma_y(x-y) = \frac{1}{1-\sigma_x(x-y)} 
\quad \mbox{for all $x,y\in\inte(K)$}.
\end{equation}

At every point $x\in\inte(K)$, a local \emph{Minkowski norm} is defined as
$|h|_x := \max{\{ \sigma_x(h), \sigma_x(-h)\}}$.
In this paper, $\sigma_x(\cdot)$ and $|\cdot|_x$ always refer to the 
primal cone $K$.   Therefore, no special notation is needed to distinguish 
these functions from the Minkowski gauge and norm associated with the dual 
cone $K^*$.
\end{subsection}

\begin{subsection}{Logarithmically homogeneous self-concordant barriers}
\label{subsec:LHSCB}
A \emph{$\vartheta$-logarithmically homogeneous barrier} for $K$ is a 
closed convex function $F \in \mathcal{C}^3$ (thrice continuously differentiable), with domain 
$\dom F = \inte(K)$, which satisfies 
\begin{equation} \label{e-log-hom}
 F(tx) = F(x) -\vartheta\ln(t) \quad \mbox{for all $x\in\inte(K)$ 
and all $t>0$}.
\end{equation}
Such a barrier is called \emph{self-concordant} if
\begin{equation} \label{e-sc}
  F'''(x;h) \preceq 2\|h\|_x F''(x) \quad\mbox{for all $x\in\inte(K)$
and all $h\in \R^n$}, 
\end{equation}
where $F''(x)$ is the Hessian at $x$, $F'''(x)$ the third derivative, and   
\begin{equation} \label{e-local-norm}
 \|h\|_x := \sqrt{\langle F''(x)h,h\rangle}.
\end{equation}
On the left-hand side of~\eqref{e-sc}, $F'''(x;h)$ is the directional 
third derivative of $F$ at $x$, in the direction~$h$, 
\[
 F'''(x;h) := \left. \frac{d}{d\alpha} F''(x+\alpha h) \right|_{\alpha =0}.
\]
In later sections, we will refer to the following properties of 
$\vartheta$-logarithmically homogeneous self-concordant barriers 
($\vartheta$-LHSCBs) \cite{NN1994,Nesterov2018}. For a more detailed
exposure to analysis of interior-point algorithms based on
general self-concordant barriers, 
see \cite{NN1994,Renegar2001,Nesterov2018}.
\begin{itemize}
\item Logarithmic homogeneity~\eqref{e-log-hom} implies the identities
\begin{equation} 
\label{e-log-hom-props}
 F'(tx) = \frac{1}{t} F'(x), \qquad 
 F''(tx) = \frac{1}{t^2} F''(x) 
\end{equation}
for all $x\in\inte(K)$ and $t>0$, and
\begin{equation} 
\label{e-log-hom-props-2}
 \langle F'(x), x\rangle = -\vartheta, \qquad
 F''(x)x = -F'(x), \qquad  F'''(x;x) = -2F''(x)
\end{equation}
for all $x\in\inte(K)$. 

\item The negative gradient at $x\in\inte(K)$ is in the interior of $K^*$:
$-F'(x) \in \inte(K^*)$ for all $x\in\inte(K)$. Moreover, $-F': \inte(K) \to \inte(K^*)$
is a bijection (see, for instance~\cite[Proposition 5.1]{Tuncel1998}).

\item The Hessian $F''(x)$ is positive definite for all $x\in\inte(K)$.
Therefore the function $\|\cdot\|_x$ defined in~\eqref{e-local-norm} is 
a norm.

\item The \emph{Dikin ellipsoid} centered at $x$ and with radius one,
is contained in $K$:    
\[
 E_{x,1} \subseteq K,
\]
where $E_{x,r} := \{ y \in \R^n : \|y-x\|_x \leq r \}$ is called the Dikin ellipsoid 
with center $x$ and radius $r$.  This property can also be expressed as
\begin{equation} \label{e-dikin}
 \sigma_x(h) \leq \|h\|_x \quad \mbox{for all $x\in\inte(K)$ and
 $h\in\R^n$}.
\end{equation}
Applying this to $-h$, we also have $\sigma_x(-h) \leq \|h\|_x$
and therefore $|h|_x \leq \|h\|_x$.

\item The variation of the Hessian on $\inte(K)$ is limited by the 
inequalities
\[
(1-\|y-x\|_x)^2 F''(x) \preceq F''(y) 
\preceq \frac{1}{(1-\|y-x\|_x)^2} F''(x).
\]
These matrix inequalities hold for all 
$x,y\in\inte(K)$ with $\|y-x\|_x < 1$.
\end{itemize}

The \emph{dual barrier} associated with $F$ is defined as 
\[
  F_*(s) = \sup_{x\in\inte(K)}\left\{-\langle s,x\rangle -F(x)\right\}.
\]
If $F$ is a $\vartheta$-LHSCB for $K$, then $F_*$ is a $\vartheta$-LHSCB 
for $K^*$.  The dual barrier of the dual barrier $F_*$ is equal to the 
original barrier $F$.
The gradients and Hessians of $F$ and $F_*$ are related by the identities
\begin{equation} \label{e-dual-grad}
 F'_*(-F'(x)) = -x, \quad F'(-F_*'(s)) = -s, \quad
\end{equation}
and
\begin{equation} \label{e-dual-hess}
 F''_*(-F'(x)) = F''(x)^{-1}, \quad  F''(-F_*'(s)) = F_*''(s)^{-1}. 
\end{equation}
We often use the notation
\[
 \tilde s := -F'(x), \qquad \tilde x := -F'_*(s),
 \qquad 
 \mu := \frac{\langle s,x\rangle}{\vartheta}, \qquad
 \tilde\mu 
 := \frac{\langle \tilde s, \tilde x\rangle}{\vartheta} 
 = \frac{\langle F'(x), F_*'(s)\rangle}{\vartheta}. 
\]
The points $\tilde s, \tilde x$ are called the \emph{shadows}
of $x$ and $s$, respectively.
The notation $\tilde s, \tilde x, \mu, \tilde \mu$ is used when it is 
clear from the context what the corresponding points $x$ and $s$ are. 
Otherwise, we add $s$ and $x$ as explicit arguments of functions
$\mu(x,s)$ and $\tilde \mu(x,s)$ defined as
\[
 \mu(x,s) := \frac{\langle s,x\rangle}{\vartheta}, \qquad
 \tilde\mu(x,s) := \frac{\langle F'(x), F_*'(s)\rangle}{\vartheta}. 
\]
\end{subsection}

\begin{subsection}{Central points and proximity measures}
Points $x\in\inte(K)$ and $s\in \inte(K^*)$ form a \emph{central pair}
if $s$ is a multiple of $-F'(x)$. 
Using~\eqref{e-log-hom-props-2}, we see that the multiple is given by the 
positive quantity $\mu$ defined above. Thus, $x,s$ form a central pair if and only
if $s = -\mu F'(x) = \mu \tilde s$.
From~\eqref{e-dual-grad} and logarithmic homogeneity, 
the following four statements are equivalent (the last one is proved in
Lemma~\ref{lem:intro1}):
\[
 x = \mu \tilde x, \qquad \tilde s = \tilde \mu s, \qquad 
 \tilde x = \tilde\mu x, \qquad \mu \tilde{\mu}=1.
\]
The motivation for this terminology is that
if in addition $x$ and $s$ are feasible in their respective optimization problems,
then $(x,s)$ is a central pair if and only if $(x,s)$ lies on the 
\emph{primal--dual central path} (with the central path parameter set to $\mu(x,s)$). 

Proximity measures that measure the deviation from centrality appear 
throughout the literature on interior-point methods.
In this paper, the \emph{gradient proximity measure} 
$\gamma_\mathrm{G}(x,s)$ and 
the \emph{uniform proximity measure} $\gamma_\infty(x,s)$,
which are defined as
\begin{equation} \label{e-gamma-g}
 \gamma_\mathrm G(x,s) := \vartheta (\mu\tilde\mu -1), \qquad
 \gamma_\infty(x,s) := \sigma_x(\mu\tilde x) -1 
\end{equation}
will be important.  
The suitability of $\gamma_\mathrm G$ and $\gamma_\infty$ as proximity
measures is stated in the following lemma.
\begin{lem}
\label{lem:intro1}(Nesterov and Todd~\cite[pp. 343--344]{NT1998}, 
 \cite[Lemma 4.1]{Tuncel2001})
If $F$ is a $\vartheta$-LHSCB for the regular cone $K$, then
\[
\vartheta \gamma_\infty(x,s) \geq \gamma_\mathrm G(x,s) \geq 0, 
 \quad \forall x \in \inte(K), \,\, \forall s \in \inte(K^*).
\]
Moreover, $\gamma_\infty(x,s) = \gamma_\mathrm G(x,s) = 0$ if and
only if $x,s$ are a central pair.
\end{lem}
For future reference, we add that the uniform proximity measure can also 
be written as
\begin{equation}
\label{e-sigma-ineqs-2}
\gamma_\infty(x,s) = \sigma_x(\mu\tilde x-x)
 = \frac{\sigma_{\mu\tilde x}(\mu\tilde x - x)}
 {1-\sigma_{\mu\tilde x}(\mu\tilde x-x)}.
\end{equation}
The first equality follows from 
\[
\sigma_x(\mu\tilde x -x)
= \inf{\{ \beta\geq 0 : (\beta +1)x-\mu\tilde x \in K\}}
= \sigma_x(\mu\tilde x) -1,
\]
since $\sigma_x(\mu\tilde x) \geq 1$.
The second equality in \eqref{e-sigma-ineqs-2} follows 
from~\eqref{e-mink-translate}.
\end{subsection}

\begin{subsection}{Primal--dual scalings}
\label{subsec:1.1}
In this section we review some definitions and results
from~\cite{Tuncel2001, MT2014}.
For each pair $x\in\inte(K)$, $s\in\inte(K^*)$, a set of local 
primal--dual scalings is defined as
\begin{equation}
\cT(x,s;\xi) := 
\left\{ T\in \SS_{++}^n : Tx = s, \; T\tilde x = \tilde s, \;
 \frac{1}{\xi\delta_F(x,s)} F''(x) \preceq T \preceq 
  \xi\delta_F(x,s) F_*''(s)^{-1} \right\},
\label{e-T}
\end{equation}
where $\xi$ is a parameter and
\begin{equation} \label{e-delta-F}
\delta_F(x,s)
:= \frac{\gamma_\mathrm G(x,s)+1}{\mu}
= \frac{\vartheta(\mu\tilde{\mu} -1)+1}{\mu}
= \iprod{F'(x)}{F'_*(s)} - \frac{\vartheta(\vartheta-1)}{\iprod{s}{x}}.
\end{equation}
This definition of $\cT$ corresponds to the set $\mathcal{T}_2$ from these 
earlier papers,
where the subscript 2 was added to highlight the usage of second-order 
information on the barrier functions $F$ and $F_*$.

The smallest value of $\xi$ for which the set $\cT(x,s;\xi)$ is nonempty 
for every pair $x\in\inte(K)$ and $s\in\inte(K^*)$ is denoted by $\xi_F$:
\[
\xi_F := \inf 
\left\{\xi : {\mathcal{T}}(x,s;\xi) \neq \emptyset,  
\forall (x,s) \in \inte(K) \oplus \inte(K^*) \right\}.
\]
The importance of these definitions lies in the complexity analysis 
in~\cite{Tuncel2001}.   If $\xi_F$ is bounded above by an absolute
constant (i.e., independent of the data of the problem instance,
dimension of the cone $K$, and  the parameter $\vartheta$ of the barrier 
function $F$),
then the second-order  algorithms in~\cite{Tuncel2001} and many other related algorithms
based on primal--dual scalings are guaranteed to generate an $\epsilon$-optimal 
solution in $O\left(\vartheta^{1/2} \ln(1/\epsilon)\right)$ iterations.
One of the key results in the theory of self-scaled barriers is the
following.
\begin{thm}(Nesterov and Todd~\cite{NT1997}, see Theorem 6.1 in~\cite{Tuncel2001})
\label{thm:intro1}
For every symmetric cone $K$ and for every $\vartheta$-self-scaled barrier 
$F$ for $K$, we have
\[
\xi_F \leq \frac{4}{3}.
\]
Moreover, the Nesterov--Todd scaling is in $\cT(x,s;4/3)$ for every pair 
$x,s\in\inte(K)$. 
\end{thm}

A unique property of the Nesterov--Todd scaling is that it has
the form $T = F''(w)$ for some (unique) point $w\in\inte(K)$.  For non-symmetric
cones, no scaling of the form $T = F''(w)$ exists that satisfies the two 
requirements 
\begin{equation} \label{e-Tx}
  Tx = s, \qquad T\tilde x= \tilde s
\end{equation}
(see \cite[section 6]{MT2014}).
Motivated by this fact, the papers \cite{Tuncel2001, MT2014} describe several methods to construct 
scalings that satisfy these conditions but are not of the form $T=F''(w)$.
For our purposes the \emph{integral scaling} from \cite[section 3]{MT2014}
will be the most relevant.  The integral scaling is defined as
\begin{equation} \label{e-T-integral}
 T := \mu \int_0^1 F''((1-\alpha)x  + \alpha (\mu\tilde x)) \, d\alpha.
\end{equation}
For a proof that this integral satisfies~(\ref{e-Tx}), 
see \cite[Theorem 3.1]{MT2014}.
\end{subsection}

\end{section}

\begin{section}{Barriers with negative curvature}
\label{sec:negative_curvature}

Theorem~\ref{thm:intro1} states that $\xi_F \leq 4/3$ for self-scaled
barriers $F$ and that the Nesterov--Todd scaling is in $\cT(x,s;4/3)$ 
for all $x, s \in \inte(K$).
In this section we extend this result to self-concordant barriers
with negative curvature.
These barriers form a special class of self-concordant barrier functions 
that strictly includes the self-scaled barriers and are defined for
convex cones that are strictly more general than homogeneous cones.
In Sections~\ref{s-neg-curv-def}--\ref{s-neg-curv-prox},
we define negative curvature barriers and discuss the most 
important properties.  
Section~\ref{s-neg-curv-xi} contains the main result on 
the complexity measure and primal--dual metrics.

\begin{subsection}{Negative curvature} \label{s-neg-curv-def}
\begin{dfn} \label{d-neg-curv}
A logarithmically homogeneous barrier $F$ for a regular convex cone $K$
has \emph{negative curvature} if for every $u \in K$, the function
\begin{equation} \label{e-neg-curv-fu}
f_u(x) := \langle F'(x), u \rangle,
\end{equation}
with $\dom f_u = \inte(K)$, is concave in $x$.
\end{dfn}
Several equivalent definitions follow from the various 
characterizations of concavity of the function~$f_u$.
For example, Jensen's inequality
\[
 f_u(\alpha x+ (1-\alpha) y) \geq 
 \alpha f_u(x) + (1-\alpha) f_u(y) 
\quad \forall u\in K, \; \forall x,y\in\inte{K}, \; \forall \alpha\in [0,1],
\]
is equivalent to the vector inequality
\begin{equation} \label{e-neg-curv-jensen}
 F'(\alpha x + (1-\alpha)y) \succeq_{K^*} 
 \alpha F'(x) + (1-\alpha) F'(y) 
\quad \forall x,y\in\inte{K}, \; \forall \alpha\in [0,1].
\end{equation}
Hence negative curvature means that the gradient of $F$
is $K^*$-concave.

Negative curvature barriers were studied by G\"{u}ler in \cite{Guler1997}
and by Nesterov and Tun\c{c}el 
in \cite{NT2016}.
They include the self-scaled barriers for symmetric 
cones~\cite[Lemma 3.1]{NT1997}, but are much more general.  
For example, composition with a linear mapping preserves
negative curvature, but in general not self-scaledness.
Therefore, the barrier
\[
 F(x) := -\ln\det(x_1A_1 + x_2A_2 + \cdots + x_nA_n),
\]
where $A_1,\ldots,A_n \in \SS^m $ are linearly independent, 
is an $m$-logarithmically homogeneous barrier with negative curvature
for the regular cone 
$K := \{x \in \R^n : x_1A_1 + x_2A_2 + \cdots + x_n A_n \succeq 0\}$
(assuming that there exists $\bar{x} \in \R^n$ such that $\sum_{j=1}^n \bar{x}_j A_j  \in \SS^m_{++}$).
\end{subsection}

\begin{subsection}{Bounds on derivatives and function values}
From the values and derivatives of a negative-curvature barrier $F$ 
at a point $x\in\inte(K)$, one obtains bounds on the values and 
derivatives at other points in the interior of $K$.
For the second derivatives, this property is referred to as the 
\emph{long-step Hessian estimation property} in \cite{MT2014}.
The inequalities listed in the next proposition are known 
to hold for self-scaled barriers for symmetric cones
\cite{NT1997,NT1998}, but are derived here under the weaker assumptions 
of logarithmic homogeneity and negative curvature. See, also \cite[Section 7]{Guler1997}.

\begin{prop}
Let $F$ be a logarithmically homogeneous barrier for the regular cone $K$.
If $F$ has negative curvature, then the following properties hold.
\begin{itemize}
\item For all $x\in\inte(K)$ and all $h$,
\begin{equation} \label{e-neg-curv-f'''}
-2 \sigma_x(h) F''(x)  \preceq F'''(x; h) \preceq 2 \sigma_x(-h) F''(x).
\end{equation}

\item
For all $x,y\in\inte(K)$,
\begin{equation} \label{e-neg-curv-hess}
\frac{1}{(1+\sigma_x(y-x))^2} F''(x) \preceq F''(y) \preceq
\frac{1}{(1-\sigma_x(x-y))^2} F''(x).
\end{equation}

\item 
For all $x,y\in\inte(K)$,
\begin{equation} \label{e-neg-curv-grad}
\frac{\|y-x\|_x^2}{1+\sigma_x(y-x)} 
\leq \langle F'(y)-F'(x), y-x \rangle \leq
\frac{\|y-x\|_x^2}{1-\sigma_x(x-y)}.
\end{equation}
\end{itemize}
\end{prop}
Note that $0 \leq \sigma_x(x-y) < 1$ for all $x,y\in \inte(K)$, so the 
bounds in~\eqref{e-neg-curv-hess} and~\eqref{e-neg-curv-grad} are well 
defined.

\begin{proof}
We start with the second inequality in~(\ref{e-neg-curv-f'''}). 
Suppose $x\in\inte(K)$.  By definition of $\sigma_x(-h)$, the vector
$u = \sigma_x(-h)x + h$ is in $K$.  
Since $F$ has negative curvature, the function 
$f_u = \langle F'(\cdot), u\rangle$ is concave.  The Hessian 
$f_u''(x) = F'''(x;u)$ at $x$ satisfies
\[
0 \succeq F'''(x; u)  
  = \sigma_x(-h) F'''(x; x) + F'''(x; h)  
  = -2\sigma_x(-h) F''(x) + F'''(x; h).
\]
The last step uses $F'''(x;x) = -2F''(x)$, which follows from logarithmic 
homogeneity (see~\eqref{e-log-hom-props-2}).
The first inequality in~(\ref{e-neg-curv-f'''}) is the second inequality 
applied to $-h$, since $F'''(x;-h) = -F'''(x;h)$.

To prove~(\ref{e-neg-curv-hess}), we define $h:=x-y$ and 
$g(\alpha) := \langle F''(x-\alpha h) v, v\rangle$ 
for fixed nonzero $v$.   From~(\ref{e-neg-curv-f'''}),
\[
\frac{-2\sigma_x(-h)}{1+\alpha \sigma_x(-h)} =
-2\sigma_{x-\alpha h}(-h) 
 \leq \frac{g'(\alpha)}{g(\alpha)} \leq 
 2\sigma_{x-\alpha h}(h)
= \frac{2\sigma_x(h)}{1-\alpha \sigma_x(h)}.
\]
The right-hand side and left-hand side both hold for all
nonnegative $\alpha$ such that $\alpha\sigma_x(h) < 1$.
Integration gives
\[
-2\ln(1+\alpha\sigma_x(-h))
\leq \ln\frac{g(\alpha)}{g(0)} \leq -2\ln(1-\alpha\sigma_x(h)).
\]
For $\alpha=1$ and $h=x-y$ (since $y \in\inte(K)$, we have
$\sigma_x(h) < 1$), this proves~(\ref{e-neg-curv-hess}).

Next, to show~(\ref{e-neg-curv-grad}) we use the expression
\[
-\langle F'(y) - F'(x), h \rangle 
 = \int_0^1 \langle F''(x-\alpha h)h, h \rangle d\alpha
\]
and bound the right-hand side by integrating~(\ref{e-neg-curv-hess})
(and using $\sigma_x(\alpha v) = \alpha \sigma_x(v)$ for nonnegative 
$\alpha$):
\[
\|h\|_x^2 \int_0^1 \frac{1}{(1+\alpha \sigma_x(-h))^2} \, d\alpha
\leq -\langle F'(y) - F'(x), h \rangle 
\leq \|h\|_x^2 \int_0^1 \frac{1}{(1-\alpha\sigma_x(h))^2} \, d\alpha.
\]
The integral on the left is $1/(1+\sigma_x(-h))$ and the integral on the
right is $1/(1-\sigma_x(h))$.
\end{proof}
\end{subsection}

We mention two more results which provide bounds on Hessians of
negative-curvature barriers. 
\begin{prop}
\label{prop:4.3}(See also \cite[Corollary 7.2]{Guler1997})
Let $F$ be a logarithmically homogeneous barrier for the regular cone
$K$.  If $F$ has negative curvature, then for all $x,y\in\inte(K)$,
\begin{equation} \label{e-geneig-sigma}
 F''(y) \preceq \sigma_y(x)^2 F''(x).
\end{equation}
In particular, if $y-x \in K$, then $\sigma_y(x) \leq 1$  and therefore
$F''(y) \preceq F''(x)$.
\end{prop}
\begin{proof}
Define $h = x - \sigma_y(x)y$.
Since $-h\in K$ (by definition of $\sigma_y(x)$), we have 
$\sigma_{x}(h) = 0$.
From  the right-hand inequality of~(\ref{e-neg-curv-hess}), 
\[
\frac{1}{\sigma_y(x)^2}F''(y)
= F''(\sigma_y(x)y)  \preceq 
 \frac{1}{(1-\sigma_{x}(h))^2} F''(x) = F''(x).
\]
\end{proof}
The inequality~\eqref{e-geneig-sigma} gives a bound on the generalized  
eigenvalues of the matrix pair $F''(y)$, $F''(x)$:
\begin{equation} \label{e-gevd-sigma}
\lambda_\mathrm{max}(F''(x)^{-1}F''(y)) \leq \sigma_y(x)^2.
\end{equation}
Later, we will see that for self-scaled barriers equality holds; 
see Lemma~\ref{lem:3.5}.

\begin{prop}
\label{prop:4.4}
Let $F$ be a logarithmically homogeneous barrier for the regular cone
$K$.  For $x,y\in\inte(K)$, define
\[
 G := \int_0^1 F''((1-\alpha) x +\alpha y) \, d\alpha.
\]
If $F$ has negative curvature, then
\begin{equation} \label{e-hess-ineq-1}
\frac{1}{1+\sigma_x(y-x)} F''(x) \preceq G \preceq 
\frac{1}{1-\sigma_x(x-y)} F''(x)
\end{equation}
and
\begin{equation} \label{e-hess-ineq-2}
 (1-\sigma_x(x-y)) F''(y) \preceq G \preceq (1+\sigma_x(y-x)) F''(y).
\end{equation}
\end{prop}
\begin{proof}
We integrate the bounds in~(\ref{e-neg-curv-hess}):
\[
\left( \int_0^1 \frac{1}{(1+\alpha \sigma_x(-h))^2} d\alpha \right) F''(x) 
\preceq
G 
 \preceq \left( \int_0^1 \frac{1}{(1-\alpha \sigma_x(h))^2} d\alpha
 \right) F''(x) 
\]
where $h=x-y$.   Evaluating the integrals gives~\eqref{e-hess-ineq-1}.
Switching $x$ and $y$ in these inequalities gives
\[
 \frac{1}{1+\sigma_y(x-y)} F''(y) \preceq G \preceq 
 \frac{1}{1-\sigma_y(y-x)} F''(y).
\]
The inequalities~(\ref{e-hess-ineq-2}) follow from 
the identity~\eqref{e-mink-translate}.
\end{proof}

\begin{subsection}{Proximity measures} \label{s-neg-curv-prox}
The next proposition, which gives an inequality between the gradient and 
uniform proximity measures, was proved for self-scaled barriers 
in \cite[Theorem 4.2]{NT1998}.

\begin{prop}
Let $F$ be a $\vartheta$-LHSCB with negative curvature. Then 
for all $x\in\inte(K)$, $s\in\inte(K^*)$,
\begin{equation} \label{e-prox-bound2}
 \gamma_\mathrm{G}(x,s) \geq
\frac{\gamma_\infty(x,s)^2}{1+\gamma_\infty(x,s)}.
\end{equation}
\end{prop}
\begin{proof}
From the first inequality in~(\ref{e-neg-curv-grad}) applied to 
$y=\mu\tilde x$, we find that
\begin{equation}
\gamma_\mathrm G(x,s)
= \langle F'(\mu\tilde x) - F'(x), \mu\tilde x-x\rangle 
\geq 
\frac{\|\mu\tilde x-x\|_x^2}{\sigma_x(\mu\tilde x)} 
= 
\frac{\|\mu\tilde x-x\|_x^2}{1+\gamma_\infty(x,s)}.\label{e-gamma-G-inf}
\end{equation}
The expression for $\gamma_\mathrm{G}$ follows from
\[
\langle F'(\mu\tilde x), \mu\tilde x\rangle 
= \langle F'(x), x\rangle 
= \langle F'(\mu\tilde x), x\rangle
= -\vartheta, \qquad
\langle -F'(x), \tilde x \rangle = \vartheta\tilde\mu.
\]
Self-concordance of $F$ further implies that $\|h\|_x \geq |h|_x$ 
for all $x\in\inte(K)$ and all $h$.  Therefore
\begin{equation}
 \|\mu \tilde x- x\|_x 
 \geq |\mu \tilde x- x|_x 
 \geq \sigma_x(\mu\tilde x -x) 
 = \gamma_\infty(x,s), \label{e-gamma-G-inf-pf}
\end{equation}
by the first identity in~(\ref{e-sigma-ineqs-2}).
Combining~(\ref{e-gamma-G-inf}) and~(\ref{e-gamma-G-inf-pf}) 
gives~(\ref{e-prox-bound2}).
\end{proof}

Simpler lower bounds on $\gamma_\mathrm G$ are 
\begin{equation} \label{e-prox-bound2-simpler}
 \gamma_\mathrm G(x,s) \geq \frac{3\gamma_\infty(x,s) -1}{4}, 
 \qquad
 \gamma_\infty(x,s) \leq \frac{4\gamma_\mathrm G(x,s) +1}{3}. 
\end{equation}
The first one follows from $x^2/(1+x) \geq (3x-1)/4$, the linear
approximation of the convex function $x^2/(1+x)$ at $x=1$.
\end{subsection}

\begin{subsection}{Bounds on the complexity measures for primal--dual 
scalings} \label{s-neg-curv-xi}
We now return to the complexity parameter $\xi_F$. 
To simplify the analysis and improve our understanding of $\xi_F$, 
it will be useful to relax the set~\eqref{e-T} by omitting 
the equality constraints $Tx=s$ and $T\tilde x = \tilde s$.  Define
\[
\widecheck{\cT}(x,s;\xi) := \left\{ T\in \SS_{++}^n : 
 \frac{1}{\xi\delta_F(x,s)} F''(x) \preceq T \preceq 
  \xi\delta_F(x,s) F_*''(s)^{-1} \right\} 
\]
and
\begin{equation} \label{e-hatcheck-T}
\widecheck \xi_F 
 :=  \inf{\left\{ \xi : \widecheck{\cT}(x,s;\xi) \neq \emptyset,
 \forall x\in \inte(K), s\in \inte(K^*) \right\}}.
\end{equation}
The complexity measure $\widecheck{\xi}_F$ admits an intriguing
geometric interpretation. For every pair of primal--dual interior points
$(x,s)$, consider the smallest blow-up factor for the Dikin ellipsoid at $x$
with radius $\delta_F^2(x,s)$ so that it contains
the dual of the Dikin ellipsoid (for $F_*$) with radius one at $s$
when that dual ellipsoid is centered at $x$.
Then $\widecheck{\xi}_F$ is the supremum of all of these
smallest blow-up factors over all primal--dual pairs of interior points $(x,s)$
(i.e., the smallest blow-up factor that guarantees the above ellipsoid
containment property for all primal--dual pairs $(x,s)$).

Clearly, $\cT(x,s;\xi) \subseteq \widecheck\cT(x,s;\xi)$  and 
$\xi_F \geq \widecheck \xi_F$.
The set $\widecheck\cT(x,s;\xi)$  is nonempty if and only if 
\[
F''(x) \preceq (\xi\delta_F(x,s))^2 F_*''(s)^{-1},
\]
i.e., $\xi \delta_F(x,s)$ is greater than or equal to the square root of
the largest eigenvalue of the matrix product $F''(x)F''_*(s)$.  
Therefore, 
\begin{equation} 
\widecheck{\xi}_F = 
\sup_{x\in\inte(K), s\in \inte(K^*)} \left\{
\frac{\lambda_\mathrm{max}^{1/2}(F''(x) F''_*(s))} {\delta_F(x,s)} \right\}
\label{e-check-xi-eig} 
\leq 
\sup_{x\in\inte(K), s\in \inte(K^*)} {
\frac{ \sigma_x(\tilde x) } {\delta_F(x,s)}}.
\end{equation}
The inequality follows from~(\ref{e-gevd-sigma}).

\begin{thm}
\label{thm:hatxi4/3}
Let $F$ be a $\vartheta$-LHSCB with negative curvature.
Then $\widecheck{\xi}_F \leq 4/3$.
\end{thm}
\begin{proof}
From~(\ref{e-check-xi-eig}), 
\[
\widecheck \xi_F 
\leq \sup_{x\in\inte(K), s\in\inte(K^*)}
   \frac{\sigma_x(\tilde x)}{\delta_F(x,s)} 
= \sup_{x\in\inte(K), s\in\inte(K^*)}
   \frac{\gamma_\infty(x,s) + 1}{\gamma_\mathrm G(x,s) + 1} 
\leq 4/3.
\]
The equality follows from the definitions of $\gamma_\infty(x,s)$
in~\eqref{e-gamma-g} and of $\delta_F(x,s)$ in~\eqref{e-delta-F}.
For the second inequality  we use the bound
$\gamma_\infty(x,s) +1 \leq 4(\gamma_\mathrm G(x,s) + 1)/3$
from~\eqref{e-prox-bound2-simpler}.
\end{proof}
\end{subsection}

\begin{thm}
\label{thm:4/3}
Let $F$ be a $\vartheta$-LHSCB with negative curvature.
Then $\xi_F \leq 4/3$.
Moreover, the integral scaling~\eqref{e-T-integral} is in $\cT(x,s;4/3)$.
\end{thm}

\begin{proof}
Let  $T$ be the integral scaling~\eqref{e-T-integral} for a primal--dual
pair $x\in\inte(K)$, $s\in\inte(K^*)$.
We verify that $T\in \cT(x,s;4/3)$.
By Theorem 3.1 in~\cite{MT2014}, $T$ satisfies the two systems of linear 
equations $Tx =s$ and $T\tilde x = \tilde s$.
From the first inequality in~(\ref{e-hess-ineq-1})
and the second inequality in~(\ref{e-hess-ineq-2}), applied with
$y=\mu\tilde x$,  
\begin{equation} \label{e-integral-scaling-sandwich}
\frac{\mu}{1+\sigma_x(\mu\tilde x -x)} F''(x)
\preceq T \preceq 
\frac{1+\sigma_x(\mu\tilde x -x)}{\mu} F''(\tilde x).
\end{equation}
Here,
\[
\frac{1+\sigma_x(\mu\tilde x -x)}{\mu} 
= \frac{\gamma_\infty(x,s) + 1}{\mu} 
\leq \frac{4(\gamma_\mathrm G(x,s) + 1)}{3\mu} 
= \frac{4}{3} \delta_F(x,s).
\]
The first equality follows from~\eqref{e-sigma-ineqs-2}.
The inequality is from~\eqref{e-prox-bound2-simpler}.
\end{proof}

\end{section}

\begin{section}{Self-scaled barriers for symmetric cones}
\label{sec:symmetric}

In this section we consider self-scaled barriers for symmetric cones.
We derive the exact value of $\xi_F$ as a function of the 
barrier parameter $\vartheta$ and show that both 
the Nesterov--Todd scaling and the integral scaling are 
in $\cT(x,s;\xi_F)$.
We will see that $\xi_F \rightarrow 4/3$ as 
$\vartheta \rightarrow \infty$, so the $4/3$ upper bound 
from Theorem~\ref{thm:intro1} is asymptotically tight.
The proofs of these results give insight into the worst-case pairs $(x,s)$,
i.e., pairs with the property that $\cT(x,s;\xi)$ is empty for 
$\xi < \xi_F$.  

\begin{subsection}{Self-scaled barriers}
The main result of this section (Theorem~\ref{thm:3.5}) will use the 
following two lemmas, where the first
one gathers results on self-scaled barriers from 
Nesterov and Todd~\cite{NT1997,NT1998}.

\begin{lem}
\label{lem:3.4}
Let $K$ be a symmetric cone and let $F$ be a $\vartheta$-self-scaled 
barrier for $K$.  Then, we have the following properties.
\begin{enumerate}
\item
Let $x,w \in \inte(K)$ and $s:=F''(w)x$.
Then,
\begin{equation} \label{e-hessian-scaling}
F''(w) K = K^*, \,\,\,\,
F'(x) = F''(w) F_*'(s) \quad \mbox { and } 
\quad F''(x) = F''(w) F_*''(s) F''(w).
\end{equation}
\item For every pair $x\in\inte(K)$, $s\in\inte(K^*)$ there exists a 
 unique $w\in\inte(K)$ that satisfies $F''(w)x = s$.   
 The scaling point $w$ also satisfies $F''(w)\tilde x = \tilde s$.
\item $F$ has negative curvature.
\item
Let $w \in \inte(K)$ and $y,v \in \bd(K)$ satisfy $\iprod{F''(w)y}{v}= 0.$ Then,
\begin{equation} \label{e-complementary-dirs}
\langle F''(w+\alpha v+\beta y)v,v\rangle 
= \langle F''(w+\alpha v)v,v\rangle  \quad
\textup{ for all } \alpha, \beta \geq 0.
\end{equation}\end{enumerate}
\end{lem}
The matrix $F''(w)$ at the scaling point $w$ in the second property 
is commonly referred to as the \emph{Nesterov--Todd scaling} at $x,s$.

\begin{proof}
The gradient and Hessian scaling property was given in 
\cite[equations (3.1) and (3.2)]{NT1997}. 
The existence and uniqueness of the scaling point is proved in
\cite[Theorem 3.2]{NT1997}.
The negative curvature property is \cite[Lemma 3.1]{NT1997}.
The last property is given in \cite[Theorem 5.1]{NT1997}. \end{proof}

The next lemma strengthens the result of 
\cite[Corollary 4.1 (ii), (iii)]{NT1997}.

\begin{lem}
\label{lem:3.5}
Let $K$ be a symmetric cone and let $F$ be a $\vartheta$-self-scaled 
barrier for $K$.  
Suppose $x,w \in \inte(K)$ and $s \in \inte(K^*)$ satisfy $F''(w)x=s$.
Then we have the following properties, where $q$ is any nonzero 
vector in $\bd(K)$ that satisfies
\[
\iprod{F''(w)q}{\sigma_x(w)x-w}  = 0.
\]
(By definition of $\sigma_x(w)$ and $F''(w)K=K^*$ such a vector $q$ exists.)
\begin{enumerate}
\item  \emph{Relation between $\sigma_x(\tilde x)$ and $\sigma_x(w)$:}
\[
\sigma_x(\tilde x) = \sigma_x(w)^2, \qquad 
\iprod{F''(w)q}{\sigma_x(\tilde x)x-\tilde x}  = 0.
\]
\item \emph{Maximum generalized eigenvalue of Hessian pairs:}
\begin{equation} \label{e-eigenvalue-sigma}
\lambda_{\max}(F''(w)^{-1}F''(x))
= \lambda_{\max}(F_*''(s) F''(w))
= \lambda^{1/2}_{\max}(F_*''(s)F''(x))
= \sigma_x(\tilde x),
\end{equation}
and $q$ is a corresponding eigenvector,
\begin{equation} \label{e-eigenvector-sigma}
F''(x)q = \sigma_x(\tilde x) F''(w)q, \quad
F''(w)q = \sigma_x(\tilde x)F''(\tilde x)q, \quad
F''(x)q = \sigma_x(\tilde x)^2 F''(\tilde x)q.
\end{equation}
\end{enumerate}
\end{lem}

\begin{proof}
The first statement is a part of \cite[Lemma 3.4]{NT1998} and its proof.
The Hessian scaling property of self-scaled barriers (Lemma~\ref{lem:3.4}) 
implies that 
\begin{equation} \label{e-scaling-pf}
F''(w)^{-1}F''(x) = F_*''(s)F''(w), \quad
F_*''(s)F''(x) = \left(F''(w)^{-1}F''(x)\right)^2 =
 \left(F_*''(s)F''(w)\right)^2.
\end{equation}
Since products of positive definite matrices have real, positive 
eigenvalues, the first two equalities in~(\ref{e-eigenvalue-sigma})
follow.
Next, we prove that $\lambda_{\max}(F''(w)^{-1}F''(x))=\sigma_x(w)^2$. 
Since self-scaled barriers have negative curvature, it follows 
from~\eqref{e-gevd-sigma} that
\begin{equation} \label{e-lambda_max-sigma}
\lambda_{\max}(F''(w)^{-1}F''(x)) \leq \sigma_x(w)^2.
\end{equation}
By definition of $\sigma_x(w)$, we have $y:= \sigma_x(w)x-w\in\bd(K)$,
so there exists a nonzero $z \in \bd(K^*)$ such that 
$\iprod{z}{y}=0$. Let $q:=F''(w)^{-1}z$.
Then $y,q \in \bd(K)$ and $\iprod{F''(w)q}{y}=0$, and the fourth property 
in Lemma~\ref{lem:3.4} (with $\alpha = 0$, $\beta=1$) implies that
\[
\frac{1}{\sigma_x(w)^2} \langle F''(x)q,q\rangle = 
\langle F''(\sigma_x(w)x)q, q\rangle = 
 \langle F''(w+y)q, q\rangle = \langle F''(w)q,q\rangle.
\]
Hence
\begin{equation} \label{e-lambdamax-rr}
\frac{\iprod{F''(x)q}{q}}{\iprod{F''(w)q}{q}} = \sigma_x(w)^2,
\end{equation}
which shows that equality holds in~(\ref{e-lambda_max-sigma}) and
that $q$ is a generalized eigenvector of the matrix pair $F''(x), F''(w)$,
i.e., the first equality in~(\ref{e-eigenvector-sigma}).
The other two equalities follow from~\eqref{e-scaling-pf}.
\end{proof}

Lemma~\ref{lem:3.5} gives a simplification of the expression for 
$\widecheck{\xi}_F$ in~(\ref{e-check-xi-eig}).  This is stated in 
the next theorem.  
The second equality in~\eqref{e-thm-3.4} should be compared with the 
inequality~\eqref{e-check-xi-eig} for barriers with negative curvature.

\begin{thm}
\label{thm:3.5}
Let $K$ be a symmetric cone and $F$ be a $\vartheta$-self-scaled barrier for $K$.  Then,
\begin{equation} \label{e-thm-3.4}
\xi_F = \widecheck \xi_F 
= \sup_{x\in\inte(K), s\in \inte(K^*)}{\frac{\sigma_x(\tilde x)} {\delta_F(x,s)}}
= \sup_{x,w\in\inte(K)}{\frac{\sigma_x(w)^2} {\delta_F(x,F''(w)x)}}.
\end{equation}
Moreover, for every pair of interior points $x\in\inte(K)$, 
$s\in\inte(K^*)$, the Nesterov--Todd scaling and the integral
scaling~(\ref{e-T-integral}) at $(x,s)$ are in $\cT(x,s;\xi_F)$.
\end{thm}
\begin{proof}
Let $K$ be a symmetric cone and $F$ be a $\vartheta$-self-scaled barrier 
for $K$. 
The second and third equations in~\eqref{e-thm-3.4} follow by 
substituting $\sigma_x(\tilde x)$ and $\sigma_x(w)^2$ for 
the root of the maximum eigenvalue in~(\ref{e-check-xi-eig}), 
using Lemma~\ref{lem:3.5}.

Next, we prove that $\xi_F = \widecheck \xi_F$ and that the 
Nesterov--Todd scaling at $(x,s)$ is in $\cT(x,s;\xi_F)$.
Since $\xi_F \geq \widecheck\xi_F$, both claims are proved if we 
show that the Nesterov--Todd scaling is in  $\cT(x,s;\widecheck \xi_F)$.
Let $T=F''(w)$ be the Nesterov--Todd scaling at $x,s$.
The equality constraints $T x  = s$ and $T \tilde x  = \tilde s$
in the definition of $\cT(x,s;\widecheck \xi(T))$ follow from
Lemma~\ref{lem:3.4}, part (2).  The two matrix inequalities
\begin{equation} \label{e-nt-psd-lmis}
\frac{1}{\widecheck \xi_F\delta_F(x,s)} F''(x) \preceq
  T \preceq \widecheck \xi_F \delta_F(x,s) F_*''(s)^{-1},
\end{equation}
for $T=F''(w)$ follow from Lemma~\ref{lem:3.5} and the definition of 
$\widecheck \xi_F$ in \eqref{e-check-xi-eig} 
(see also \cite[Corollary 4.1]{NT1997}). 

To see that the integral scaling~(\ref{e-T-integral}) is also 
in~$\cT(x,s;\xi_F)$, we strengthen the result in 
Theorem~\ref{thm:4/3}.  As noted in the proof of that theorem,
the integral scaling satisfies $Tx=s$ and $T\tilde x= \tilde s$,
and the inequalities~(\ref{e-integral-scaling-sandwich}).
Since $\sigma_x(\mu\tilde x -x) = \mu \sigma_x(\tilde x)$ 
(see~\eqref{e-sigma-ineqs-2}), the 
inequalities~(\ref{e-integral-scaling-sandwich}) can be written as 
\[
\frac{1}{\sigma_x(\tilde x)} F''(x) \preceq T \preceq 
\sigma_x(\tilde x) F''(\tilde x).
\]
By~\eqref{e-thm-3.4} we have
$\sigma_x(\tilde x) \leq  \widecheck \xi_F \delta_F(x,s)$.
Hence the inequalities~(\ref{e-nt-psd-lmis}) hold for the integral 
scaling~$T$.
\end{proof}
In the following sections we work out the exact value of $\xi_F$
for commonly used self-scaled barriers.  To simplify the notation,
we define
\begin{equation} \label{e-tau-rho}
\tau_n := \sqrt{\frac{n}{n-1}}, \qquad
\rho_n := \frac{(\tau_n +1)^2}{\tau_n (\tau_n +2)},
\end{equation}
for integers $n\geq 2$.
Figure~\ref{f-xi-hat-check} shows $\rho_n$ as a function of $n$.
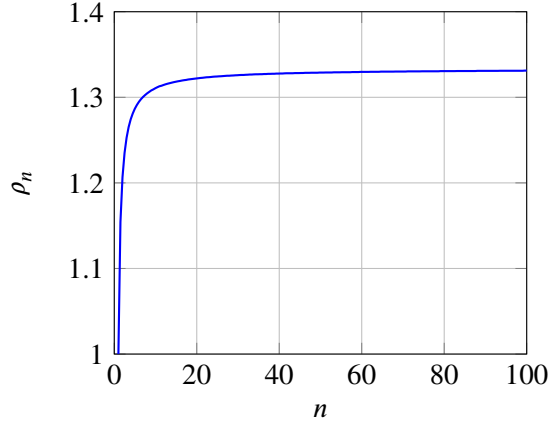
\begin{figure}
\small
\centering
\begin{tikzpicture}
\begin{axis}[
  scale = .8,
  xmin = 0, xmax = 100,
  ymin = 1, ymax = 1.4,
  xlabel = $n$,
  ylabel = $\rho_n$,
  grid,
  clip = false,
]
\addplot[no markers, thick, blue, domain = 1:100, samples = 200] 
  {1/( 1 - (x-1) * (sqrt(x) - sqrt(x-1))^2)};
\end{axis}
\end{tikzpicture}
\caption{The function $\rho_n$ defined in~\eqref{e-tau-rho}.
The limit for $n\rightarrow\infty$ is $4/3$.}
\label{f-xi-hat-check}
\end{figure}
\end{subsection}

\begin{subsection}{Nonnegative orthant}
We start with the nonnegative orthant ($K = K^* = \R^n_+$) 
and the logarithmic barrier function $F:\R^n \to \R \cup\{+\infty\}$ 
defined by
\begin{equation} \label{e-nn-orthant-F}
F(x) :=\left\{ \begin{array}{rl} -\sum_{j=1}^n \ln x_j, & 
\textup{ if } x \in \R_{++}^n,\\
+\infty, & \textup{ otherwise.}
\end{array}
\right.
\end{equation}
This is a self-scaled barrier with $\vartheta = n$.
The dual barrier is
\[
F_*(s) :=\left\{ \begin{array}{rl} -\sum_{j=1}^n \ln s_j - n, & 
\textup{ if } s \in \R_{++}^n,\\
+\infty, & \textup{ otherwise.}
\end{array}
\right.
\]

\begin{thm}
\label{thm:2}
Let $F$ be the barrier function~\eqref{e-nn-orthant-F}, with $n\geq 2$.
Then $\xi_F = \widecheck{\xi}_F = \rho_n$.
Moreover, for every pair $x,s\in\R^n_{++}$, the diagonal matrix 
$T$ with diagonal entries $T_{ii} = s_i/x_i$ is in $\cT(x,s;\xi_F)$.
\end{thm}

\begin{proof}
Let $x,s \in \R^n_{++}$.  Then
\[
\sigma_x(\tilde x) = 
 \frac{1}{\min_i \left\{x_i s_i\right\}}, \qquad
\delta_F(x,s) = n\tilde \mu - \frac{n-1}{\mu}
 = \sum_{i=1}^n \frac{1}{x_is_i}  - \frac{n(n-1)}{\sum_i x_i s_i}
\]
and therefore
\begin{equation}
\label{e-check-F-pf}
\frac{1}{\widecheck \xi_F} = 
\inf_{x,s\in\R^n_{++}} 
{\left(\min_{i \in\{1,\ldots,n\}} v_i
   \left( \sum_{j=1}^n \frac{1}{v_j} - \frac{n(n-1)}
 {\sum_j v_j }\right)\right)},
\end{equation}
where $v\in\R^n_{++}$ is the vector
with components $v_i = x_is_i$ for $i \in \{1,\ldots, n\}$.
(The vector $v$ is the component-wise square of the usual $v$-space vector 
from the interior-point literature.)
To solve the minimization problem in this expression, we can
assume, without loss of generality, that  $v_n = \min_i{v_i} = 1$.
Therefore, $1/\widecheck \xi_F$ is the optimal value of the
optimization problem
\begin{equation} \label{e-check-F-pf-2}
\begin{array}{ll}
\mbox{minimize}   & \displaystyle
   1 + \left(\sum\limits_{j=1}^{n-1} \frac{1}{v_j}\right) - \frac{n-1}{\alpha}  \\
\mbox{subject to} & v_i \geq 1, \quad i \in \{1, 2, \ldots, n-1\} \\
  & 1 + \sum\limits_{i=1}^{n-1} v_i = n\alpha,
\end{array}
\end{equation}
with variables $v_1,\ldots,v_{n-1}$ and $\alpha$.
We solve the problem with the inequalities omitted and will find that they
are inactive at the optimum.  For fixed $\alpha$, the problem is convex
and the optimality conditions are
\[
\frac{1}{v_i^2} = \lambda, \quad i \in \{1,\ldots,n-1\}, \qquad
 1 + \sum_{i=1}^{n-1} v_i = n\alpha,
\]
where $\lambda$ is a Lagrange multiplier for the equality constraint.
Hence, for given $\alpha$, the variables $v_1,\ldots,v_{n-1}$ are equal
at the optimum.  From the equality constraint,
$v_i = \tau_n^2 (\alpha-1) + 1$ for $i \in \{1,\ldots,n-1\}$.
Substitution in the objective function of~\eqref{e-check-F-pf-2} gives
the optimal value as a function of $\alpha$,
\[
 f(\alpha) := 1 + \frac{n-1}{\tau_n^2(\alpha-1) +1} - \frac{n-1}{\alpha}.
\]
This is minimized by $\alpha = 1 + 1/\tau_n$
and the optimal value is $f(1+1/\tau_n) = 1/\rho_n$.  
We conclude that the optimal
value and optimal solution of~\eqref{e-check-F-pf-2} are
\begin{equation}
\label{e-v-opt}
\frac{1}{\widecheck \xi_F} 
= \inf_\alpha{f(\alpha)} = \frac{1}{\rho_n}, \qquad
 v = (\tau_n+1, \ldots, \tau_n+1, 1).
\end{equation}
The last statement in the theorem follows from Theorem~\ref{thm:3.5} and
the fact that the diagonal matrix with diagonal entries $s_i/x_i$
is the Nesterov--Todd scaling for the barrier $F$.
\end{proof}
The proof established that the value of $\widecheck{\xi}_F$ is attained  
by pairs $(x,s)$ with component-wise product $v$ given by~\eqref{e-v-opt}.
For these worst-case pairs, $\mu = (\tau_n+1)/\tau_n$.
Other worst-case pairs are constructed by permuting and scaling
the solution $v$ in~\eqref{e-v-opt}.  

\begin{cor}
\label{cor:2.1}
Let $F$ be the barrier function~\eqref{e-nn-orthant-F}, with $n\geq 2$.
Then for every $s \in \R^n_{++}$, there is an 
$x \in \R^n_{++}$ which together with $s$ attains 
$\widecheck{\xi}_F$. In particular,
for each $\mu > 0$, pick any $i \in \{1,2, \ldots,n\}$ and set
\begin{equation} \label{e-Rn-zi}
x_j := \left\{ \begin{array}{rl} 
(\tau_n/(\tau_n+1))\mu \tilde x_i,& \textup{ if } j = i;\\
\tau_n \mu \tilde x_j, & \textup{ otherwise,} 
\end{array}
\right.
\end{equation}
where $\tilde x_j = 1/s_j$ for $j \in \{1, 2, \ldots,n\}$.
\end{cor}
It can be verified that for all worst-case pairs described in the corollary,
\[ 
\gamma_\mathrm{G}(x,s) = \frac{1}{\tau_n+1}, \qquad
\gamma_\mathrm{\infty}(x,s) = \frac{1}{\tau_n}, \qquad
\|x-\mu\tilde x\|_{\mu\tilde x} =  \frac{\tau_n}{\tau_n+1}.
\]
Figure~\ref{fig:nonnegorth} illustrates the corollary in dimension $n=3$.
\begin{figure}
  \includegraphics[width=0.7\textwidth, trim=3cm 3cm 3cm 3cm]{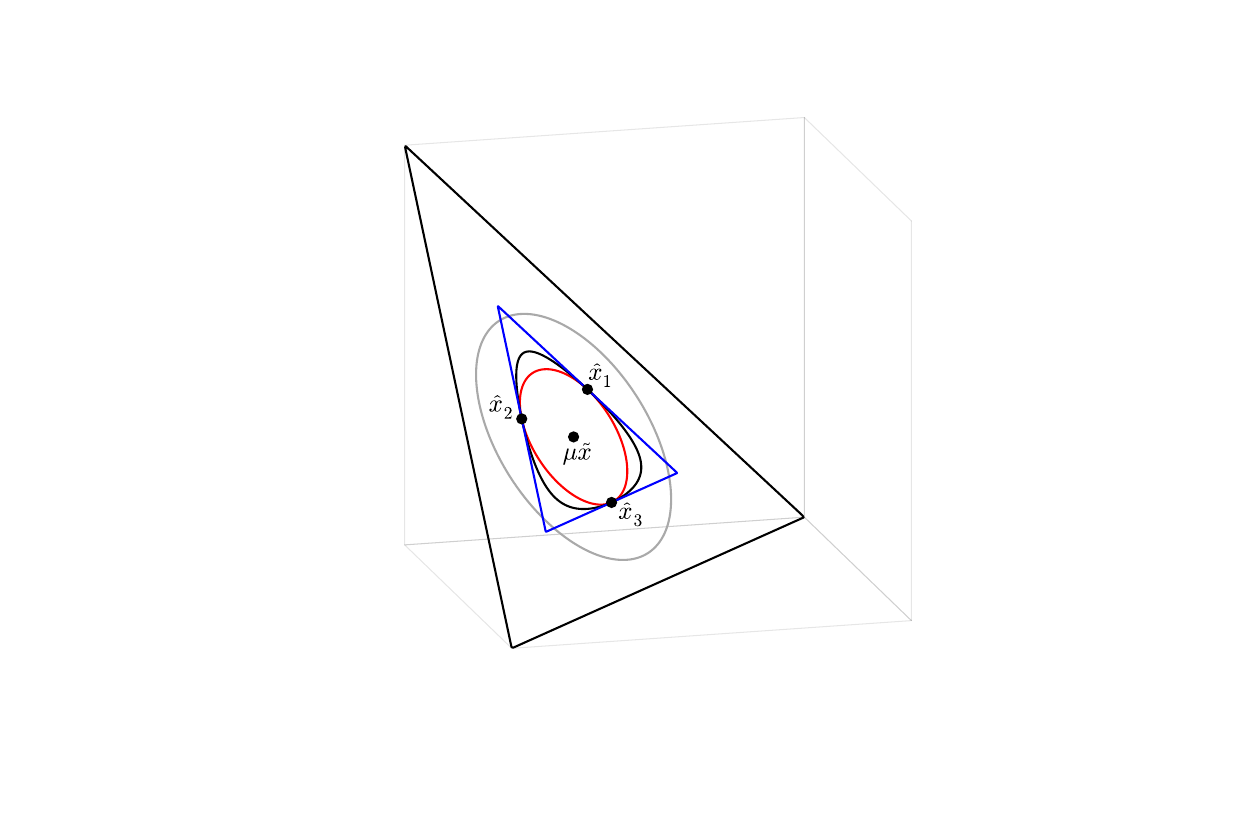}
  \caption{Nonnegative orthant of $\R^n$ for $n=3$ intersected with a  
  hyperplane\\ $\{x \in \R^n\, : \, \iprod{s}{x}= n\mu\}$ where $s$ is
  a positive vector.  The gray curve is the boundary of the unit Dikin 
  ellipsoid centered at $\mu\tilde x$. 
  The points $\hat{x}_1, \hat{x}_2, \hat{x}_3$ are the three points in the
  plane that attain $\widecheck\xi_F$, as given by Corollary~\ref{cor:2.1}.  
  They lie on the boundary of the Dikin ellipsoid with radius 
  $\tau_n/(\tau_n+1)$ (red curve).  
  The black curve shows the boundary of the neighborhood
  $\gamma_\mathrm{G}(x,s)\leq 1/(\tau_n+1)$. The blue triangle is the boundary 
  of the neighborhood $\gamma_{\infty}(x,s) \leq 1/\tau_{n}$. 
  }
  \label{fig:nonnegorth}

\end{figure}
In the figure, the nonnegative orthant in $\R^3$ is intersected with 
the plane $\{ x \in \R^n : \langle s,x\rangle = n\mu\}$, for a given positive
vector $s$ and constant $\mu >0$.  
The point $\mu\tilde x$ in the plane forms a central pair with $s$.  
The two elliptical curves show the boundaries of the Dikin ellipsoids 
$E_{\mu\tilde x, r} = \{ x \in \R^n : \|x -\mu\tilde x\|_{\mu\tilde x} \leq r \}$
with radius $r=1$ (gray curve) and $r= \tau_n/(\tau_n+1)$ (red curve).
The points $\hat x_1,\hat x_2, \hat x_3$ are the three points in the plane
that attain $\widecheck\xi_F$ together with $s$.
The black curve is the contour line of 
$\gamma_\mathrm G(x,s)  = 1/(\tau_n + 1)$ and the blue curve is the contour 
line of $\gamma_\infty(x,s) = 1/\tau_n$.
\end{subsection}

\begin{subsection}{Positive semidefinite cone}
\label{s-psd}
For the cone of symmetric positive semidefinite matrices
$K = K^* = \SS^m_{++}$, the barrier
\begin{equation} \label{e-psd-F}
F(x) :=\left\{ \begin{array}{rl} -\ln\det(x), 
    & \textup{ if } x \in \SS_{++}^m;\\
+\infty, & \textup{ otherwise,}
\end{array}
\right.
\end{equation}
and the corresponding dual barrier
\[
F_*(s) :=\left\{ \begin{array}{rl} -\ln\det(s) - m, 
    & \textup{ if } s \in \SS_{++}^m;\\
+\infty, & \textup{ otherwise,}
\end{array}
\right.
\]
are self-scaled with $\vartheta = m$.  Here, the dual cone and
dual barrier are defined with respect to the trace inner product
$\langle s, x \rangle = \Tr(sx)$.

\begin{thm}
\label{cor:3.7}
Let $F$ be the barrier function~\eqref{e-psd-F}, with $m\geq 2$.
Then $\xi_F = \widecheck \xi_F = \rho_m$.
Moreover, the Nesterov--Todd scaling and the integral scaling at $x,s$ are 
in $\cT(x,s;\xi_F)$.
\end{thm}
\begin{proof}
Let $x,s\in\SS^m_{++}$.
Then
\[
\sigma_x(\tilde x) = \lambda_\mathrm{min}(xs)
\qquad
\delta_F(x,s) = \Tr(x^{-1}s^{-1}) - \frac{m(m-1)}{\Tr(xs)}.
\]
If we denote the eigenvalues of $xs$ in these expressions by $v_i$,
and apply~\eqref{e-thm-3.4} in Theorem~\ref{thm:3.5}, we obtain
\[
\frac{1}{\widecheck \xi_F} 
= \inf_{v\in\R^m_{++}} \left( (\min_i v_i) 
\left(\sum_{j=1}^m \frac{1}{v_j} - \frac{m(m-1)}{\sum_{j=1}^m v_j}
\right)\right).
\]
The rest of the proof follows as in the proof of Theorem~\ref{thm:2}.
The statement about the Nesterov--Todd scaling and the integral scaling
also follow from Theorem~\ref{thm:3.5}. 
\end{proof}

Figure~\ref{fig:banded_toeplitz2} (left) shows the cone $\SS^2_+$ intersected
with a hyperplane $\{ x \in \SS^2_+ : \Tr(sx) = 2\mu\}$, for a given positive
definite $s$ and positive $\mu$.  Three points $\hat x_i$ in the plane 
attain $\widecheck \xi_F$  together with $s$.
\begin{figure}
\centering
  \includegraphics[width=0.45\textwidth]{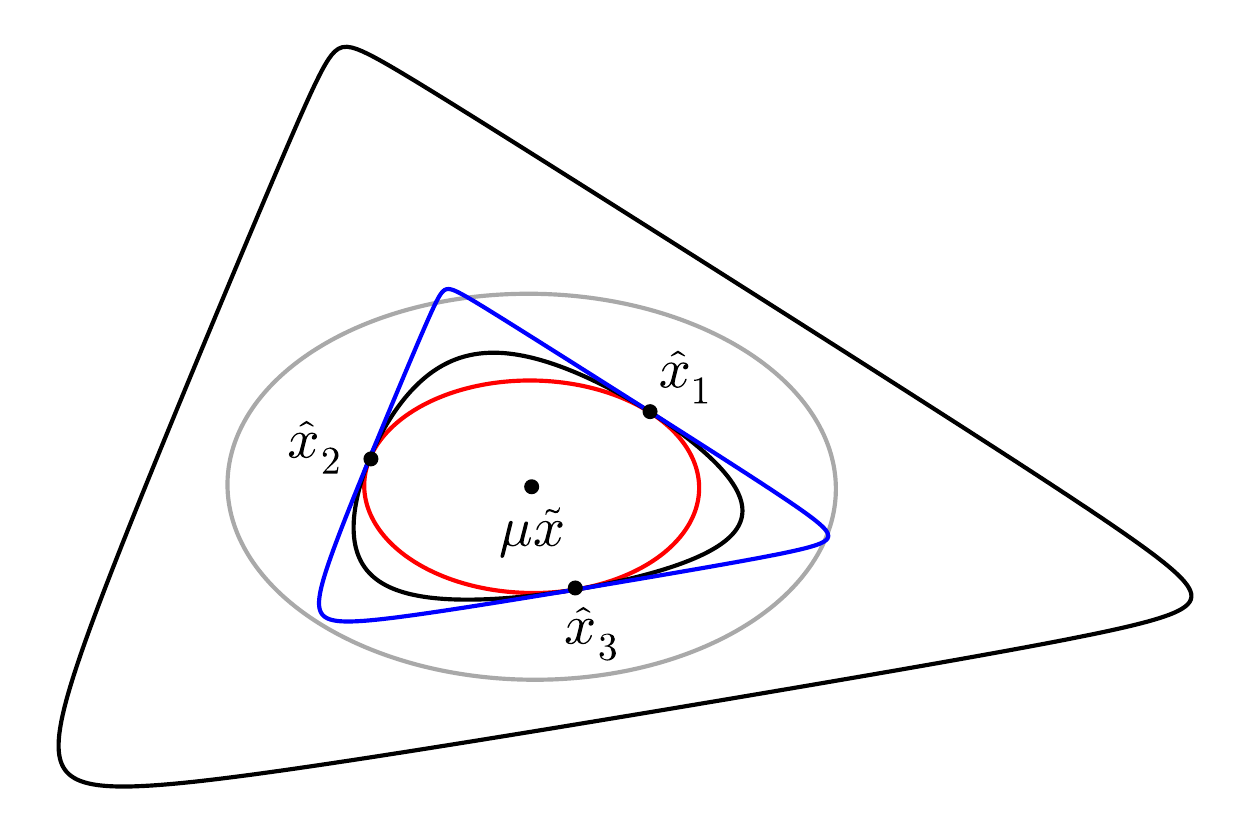}
  \includegraphics[width=0.45\textwidth]{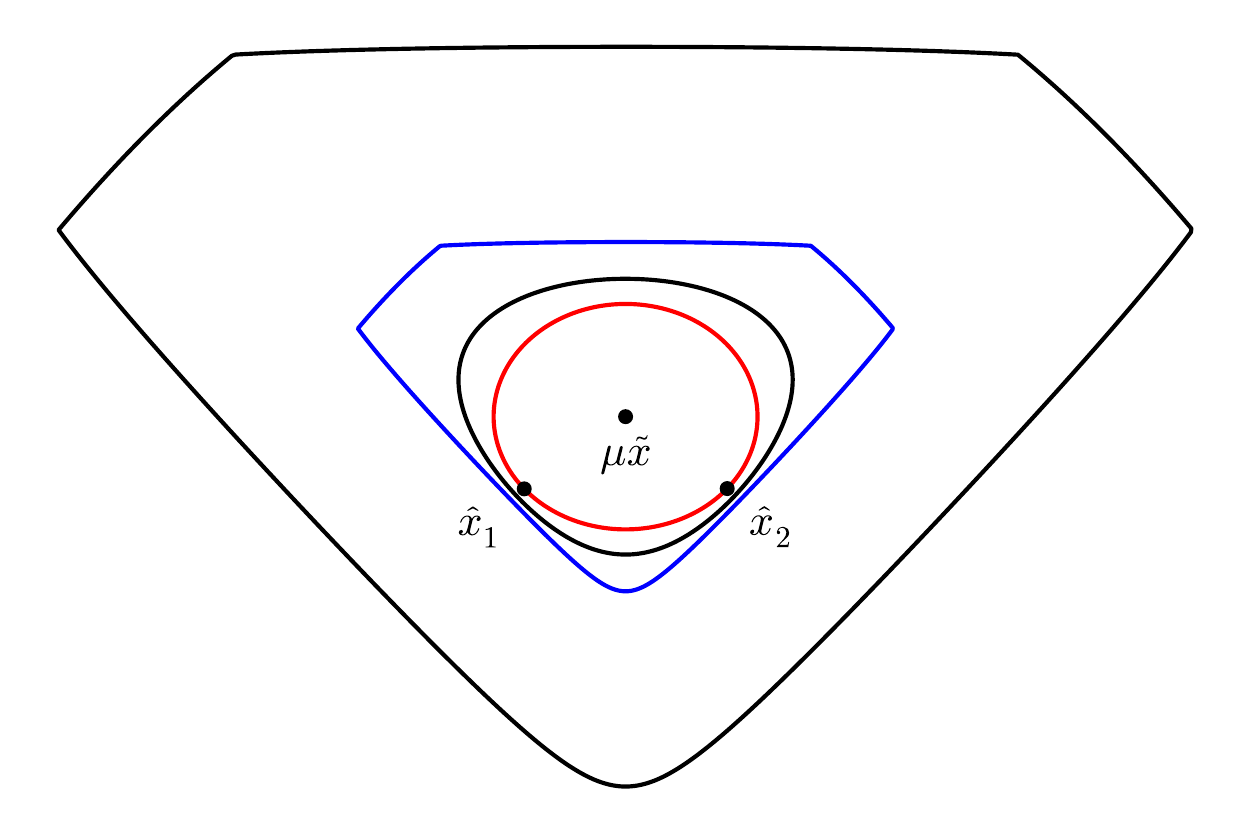}
  \caption{\emph{Left.} The intersection of the three-dimensional positive 
   semidefinite cone $\SS^2_+$ and the hyperplane defined by $\iprod{x}{s}
   = 2\mu$ for a given $s$ and $\mu$.
  \emph{Right.} Tridiagonal Toeplitz cone intersected with the 
   hyperplane
  $\{x \in \R^3\, : \, \iprod{x}{s}=\mu \vartheta\}$ for a given $s$
  and $\mu$.  The red Dikin ellipsoids have radius 
  $\tau_{\vartheta}/(\tau_{\vartheta}+1)$ and contain the
  points $\hat x_i$. The black curves show the boundary of the 
  neighborhood $\gamma_G\leq 1/(\tau_{\vartheta}+1)$ and the blue curves 
  show the boundary of the region $\gamma_{\infty}\leq 1/\tau_{\vartheta}$.}
  \label{fig:banded_toeplitz2}
\end{figure}

The figure on the right illustrates some of the differences between
self-scaled barriers and LHSCBs with negative curvature.
It shows the $m\times m$ tridiagonal positive semidefinite Toeplitz cone
(with $m=5$), intersected with the hyperplane
$\{x : \langle s, x \rangle = m\mu\}$ for a fixed $s\in\inte(K^*)$ and $\mu$.
The barrier is the standard log-det barrier with $\vartheta =m$, i.e.,
of the form $F(x) = -\ln\det(A_0 +x_1 A_1 + x_2 A_2 + x_3A_3)$ for some 
symmetric $m\times m$ matrices $A_0, A_1, A_2, A_3$. 
This barrier has negative curvature but it is not self-scaled.
Two points $\hat x_1, \hat x_2$ attain $\widecheck \xi_F$ together with $s$. 
These points were determined numerically.
In the two figures we also show the boundaries of three neigbourhoods: 
the Dikin ellipsoid with radius $\tau_\vartheta/(\tau_\vartheta  +1)$,
the $\gamma_\mathrm G$-neighbourhood with radius $1/(\vartheta +1)$,
and the $\gamma_\infty$-neighbourhood with radius $1/\vartheta$.
In the example on the right,
\[
\widecheck \xi_F 
< 
\frac{\gamma_\infty(\hat x_i, s)+1} {\gamma_\mathrm G(\hat x_i, s)+1} 
< \rho_\vartheta. 
\]
We also note that the worst-case points $\hat x_1,\hat x_2$ lie on
the boundary of the Dikin ellipsoid with 
radius $\tau_\vartheta/(\tau_\vartheta+1)$.
This will be further investigated in Section~\ref{sec:general_cones}.
\end{subsection}

\begin{subsection}{Second order cone}
Next, we apply Theorem~\ref{thm:3.5} to the logarithmic barrier
for the second order cone and direct sums of second order cones.
For the second order cone $K= K^* = \mathcal Q_p 
:= \{(y,t) \in \R^p \oplus \R : \|y\|_2 \leq t\}$, the barrier
\begin{equation} \label{e-soc-F}
F(y,t) :=\left\{ \begin{array}{rl} -\ln(t^2 - y^\top y) 
    & \textup{ if } (y,t)  \in \inte(\mathcal Q_p); \\
+\infty, & \textup{ otherwise,}
\end{array}
\right.
\end{equation}
and the corresponding dual barrier
\[
F_*(w,r) :=\left\{ \begin{array}{rl} -\ln(r^2 - w^\top w) -2 
    & \textup{ if } (w,r) \in \inte(\mathcal Q_p); \\
+\infty, & \textup{ otherwise,}
\end{array}
\right.
\]
are self-scaled with $\vartheta = 2$.  Here, the dual cone and
dual barrier are defined with respect to the Euclidean inner product
$\langle s, x \rangle := s^\top x$.
If $K$ is the direct sum of $m$ second order cones for some positive 
integer $m$, then the barrier
\begin{equation} \label{e-soc-F2}
F(y_1,t_1, \ldots, y_m, t_m) 
:=\left\{ \begin{array}{rl} -\sum\limits_{i=1}^m
  \ln(t_i^2 - y_i^\top y_i) 
    & \textup{ if } (y_i,t_i)  \in \inte(\mathcal Q_p), \; i \in \{1,\ldots,m\}; \\
+\infty, & \textup{ otherwise,}
\end{array}
\right.
\end{equation}
is self-scaled with $\vartheta = 2m$.

\begin{thm}
\label{cor:3.8}
Let $F$ be the barrier function~\eqref{e-soc-F2}.
Then $\xi_F = \widecheck \xi_F = \rho_{2m}$.
Moreover, the Nesterov--Todd scaling and integral scaling at $x,s$ are 
in $\cT(x,s;\xi_F)$.
\end{thm}

The proof is along the lines of the proofs of 
Theorems~\ref{thm:2} and~\ref{cor:3.7}. 
\end{subsection}

\begin{subsection}{General symmetric cones}
The results in the previous sections generalize to all symmetric cones 
and optimal self-scaled barriers. 
This can be shown by extending the proof of Theorem~\ref{thm:2}
using the general theory of self-scaled barriers and Euclidean Jordan algebras 
\cite{NT1997,NT1998,Tuncel1998,Fay:97,Fay:97b,AlG:03,ScA:01,ScA:03,Stu:00,
HauserGuler2002,Schmieta2000,FaK:94}.  In particular,
Hauser and G\"uler have proved the following classification result
for self-scaled barriers \cite[Theorem 5.5]{HauserGuler2002}.
Let $K = K_1 \oplus \cdots \oplus K_m$ be a symmetric cone, decomposed into 
irreducible symmetric components $K_1,\ldots,K_m$.
Then all self-scaled barriers for $K$ can be expressed as  
\begin{equation} \label{e-ss-barrier}
F(x_1,\ldots,x_m) = 
 \left\{ \begin{array}{rl}
  c_0 - \sum_{i=1}^m c_i \ln \det_{K_i}(x_i), & x_i\in\inte(K_i), \;
 i \in \{1,\ldots,m\}; \\
  + \infty, & \mbox{otherwise,} \end{array}\right.
\end{equation}
where $c_i\geq 1$ for $i \in \{1,\ldots,m\}$, and $\det_{K_i}(x_i)$ is the 
determinant of $x_i$ in the Jordan algebra associated with $K_i$.
The logarithmic homogeneity parameter for~\eqref{e-ss-barrier}
is equal to $\vartheta  := c_1r_1+\cdots +c_mr_m$, where
$r_i$ is the rank of the Jordan algebra associated with $K_i$.
The self-scaled barrier $F$ defined in~\eqref{e-ss-barrier} is 
\emph{optimal} if and only if $c_i=1$ for all $i \in \{1,\ldots,m\}$.  For an optimal self-scaled 
barrier $\vartheta = r_1 + \cdots + r_m$, the rank of the cone $K$; see~\cite{GulerTuncel1998}.

\begin{thm} \label{thm-ss-barrier}
Let $F$ be an optimal self-scaled barrier for a symmetric cone
$K$, with rank $\vartheta \geq 2$.
Then $\xi_F = \widecheck \xi_F = \rho_\vartheta$.
Moreover, the Nesterov--Todd scaling and the integral scaling at $x,s$ are 
in $\cT(x,s;\xi_F)$.
\end{thm}

\begin{proof}
Let $K$ be a symmetric cone. By the above-mentioned classification~ \cite[Theorem 5.5]{HauserGuler2002}
of optimal self-scaled barriers, 
\[
F(x) = \left\{ \begin{array}{rl}
  c_0 - \ln \det_{K}(x), & x\in\inte(K); \\
  + \infty, & \mbox{otherwise,} \end{array}\right.
\]
where $c_0$ is a constant
and $\det_K(x)$ is the determinant of $x$ in the Jordan algebra associated with $K$ (if $K$ is not
irreducible, then $\det_K$ is the product of the determinants $\det_{K_i}$).
Let $x\in\inte(K)$ and $s\in \inte(K^*)$. Denote by $v_i$ for $i \in \{1,\ldots,\vartheta\}$, the eigenvalues 
of $P(s^{1/2}) x$ where $P$ is the quadratic representation
of the Jordan algebra associated with $K$ (see \cite[Chapter 2]{FaK:94}
for the definitions of eigenvalues, trace, inverse, square root, and quadratic 
representation in Jordan algebras).

Next, we prove the key observations that
\[
 \sigma_x(\tilde x) = \frac{1}{\min\limits_{i \in \{1,\ldots,\vartheta\}} v_i}, \qquad
 \textup{ and } \qquad
  \delta_F(x,s) = \left(\sum_{i=1}^\vartheta \frac{1}v_i \right)
   - \frac{\vartheta(\vartheta-1)}{\sum_{i=1}^\vartheta v_i}.
\]
Using the definitions and Proposition III.4.2 part (ii) of~\cite{FaK:94}, we
have
\[
\sigma_x(\tilde x) = \inf \left\{\beta \geq 0 \, : \, \beta x - s^{-1} \in K \right\}.
\]
Since $P(s^{1/2}) \in \aut(K)$, and $P(s^{1/2})s^{-1}=e$ (where $e$ is the \emph{identity
element} in the underlying Jordan algebra), we obtain
\[
\sigma_x(\tilde x) = \inf \left\{\beta \geq 0 \, : \, \beta P(s^{1/2})x  - e \in K \right\},
\]
establishing $\sigma_x(\tilde x) = 1/\min_{i \in \{1,\ldots,\vartheta\}} v_i.$

Next, using the last expression in \eqref{e-delta-F}, Proposition III.4.2 part (ii) of~\cite{FaK:94},
and the facts that\\
$\iprod{s^{-1}}{x^{-1}} = \Tr\left(P(s^{-1/2})x^{-1}\right)$, 
$\iprod{s}{x} = \Tr\left(P(s^{1/2})x\right)$,
we obtain
\[
\delta_F(x,s) = \Tr\left(P(s^{-1/2})x^{-1}\right) - \frac{\vartheta(\vartheta-1)}{\Tr\left(P(s^{1/2})x\right)}.
\]
Since the trace is the sum of the eigenvalues of its argument and
the eigenvalues of $P(s^{-1/2})x^{-1}$ are the reciprocals of
the eigenvalues of $P(s^{1/2})x$, the claimed expression for
$\delta_F(x,s)$ follows from the definition of $v_i$.

Substituting these expressions in~\eqref{e-thm-3.4} shows that
$1/\widehat\xi_F$ is the solution of the same optimization problem 
\eqref{e-check-F-pf-2} as in the proof of Theorem~\ref{thm:2} 
(with $n$ replaced with $\vartheta$).
The last statement in the theorem was proved in Theorem~\ref{thm:3.5}.
\end{proof}

Note the condition in Theorem~\ref{thm-ss-barrier} that $F$
is an \emph{optimal} self-scaled barrier.  
As an example of a non-optimal self-scaled barrier, consider
$F(x_1,x_2) := -2\ln x_1 - 2\ln x_2$ for $K := \R^2_+$,
and the corresponding dual barrier 
$F_*(s_1,s_2) = -2\ln s_1 - 2\ln(s_2) -4$.
The barrier $F$ is self-scaled with $\vartheta=4$, but not optimal. 
In the expression~\eqref{e-thm-3.4} for $\xi_F$ we have  
$\tilde x = (2/s_1, 2/s_2)$, 
\[
\sigma_x(\tilde x) = \frac{2}{\min{\{x_1s_1,x_2s_2\}}}, \quad
\quad
\delta_F(x,s)
 = \frac{4}{x_1s_1} + \frac{4}{x_2s_2}
    - \frac{\vartheta(\vartheta-1)}{x_1s_1+x_2s_2}.
\]
Hence, if we define $v_i =x_is_i$, 
\[
\frac{1}{\xi_F} 
= \inf_{v_1,v_2 > 0}
 \frac{\min{\{v_1,v_2\}}}{2}
 \left( \frac{4}{v_1} + \frac{4}{v_2}
    - \frac{\vartheta(\vartheta-1)}{v_1+v_2}\right)
 = \inf_{v_1\geq 1}
 \left( \frac{2}{v_1} + 2 - \frac{\vartheta (\vartheta-1)}{2(v_1+1)} \right).
\]
This gives $\xi_F = 1.0774 < \rho_\vartheta = 1.2071$.
\end{subsection}
\end{section}

\begin{section}{General logarithmically homogeneous self-concordant 
barriers}

\label{sec:general_cones}

In Section~\ref{sec:negative_curvature} an upper bound on $\xi_F$ was derived 
for LHSCBs with negative curvature (Theorem~\ref{thm:hatxi4/3}).
In this section, we present bounds for the relaxed complexity measure $\widecheck{\xi}_F$
that apply to all LHSCBs. 
We also include examples in which the solution of the optimization problem in 
the definition of $\widecheck{\xi}_F$ in~\eqref{e-check-xi-eig}
is examined numerically.

\subsection{Bounds on $\widecheck{\xi}_F$}
In Section~\ref{sec:negative_curvature}, the upper bound $4/3$ on 
$\widecheck{\xi}_F$ and $\xi_F$ was proved for negative curvature LHSCBs.
The proof was based on Proposition~\ref{prop:4.3}. 
In this section, we relax the negative curvature assumption. Instead
of Proposition~\ref{prop:4.3},  we have a more generally applicable, but 
relatively (and somewhat significantly) weaker conclusion in 
Proposition~\ref{prop:5.2}.
\begin{prop}\cite[Corollary 1.3]{NT2016}
\label{prop:5.1}
Let $F$ be a $\vartheta$-LHSCB for $K$, and $x, u \in \inte(K)$ such that $\iprod{F'(x)}{u-x} \geq 0$. Then, $F''(x) \preceq 4 \vartheta^2 F''(u).$
\end{prop}
Using the above, we conclude the following.
\begin{prop}
\label{prop:5.2}
Let $F$ be a $\vartheta$-LHSCB for $K$, and $x \in \inte(K)$ and 
$s \in \inte(K^*)$. Then, 
\[
\frac{1}{4\vartheta^2 \mu^2}F''(\tilde{x}) \preceq  F''(x) \preceq 4 \vartheta^2 \tilde{\mu}^2F''(\tilde{x}).
\]
\end{prop}

\begin{proof}
Define $u:=\tilde x/\tilde{\mu}$. Then $\iprod{F'(x)}{u-x} = 
(\vartheta \tilde{\mu})/\tilde\mu -\vartheta= 0$ and, by 
Proposition~\ref{prop:5.1}, 
$F''(x) \preceq 4\vartheta^2 F''\left(-F'(s)/\tilde{\mu}\right)$.
This proves the right-hand side inequality in the statement. 
Next, define $u:=x/\mu \in \inte(K)$.  Then $\iprod{F'(\tilde{x})}{u-\tilde{x}} = 
(\vartheta {\mu})/\mu -\vartheta= 0$ and, 
again by Proposition~\ref{prop:5.1}, 
$F''(\tilde{x}) \preceq 4\vartheta^2 F''\left(x/\mu\right)$
proving the other inequality in the statement.
 \end{proof}
 
 Another ingredient we need in this section is next.
 \begin{lem}
 \label{lem:5.1}
 Let $F$ be a $\vartheta$-LHSCB for $K$. Then, for every $x \in \inte(K)$ and for every $u \in K$, we have
\[
\iprod{-F'(x)}{u} \leq \sqrt{\vartheta}\norm{u}_x \leq \sqrt{\vartheta}\iprod{-F'(x)}{u}.
\]
 \end{lem}
 
\begin{proof}
The left-hand inequality follows from 
$\iprod{F'(x)}{F''(x)^{-1}F'(x)} = \vartheta$ and the Cauchy--Schwarz 
inequality.
The right-hand-side is Corollary~2.3.1 in \cite{NN1994}  
(alternatively, see 
Theorem~1.1 in \cite{NT2016} or Theorem 5.1.14 in \cite{Nesterov2018}).
\end{proof}
 
\begin{thm}
\label{thm:5.1}
Let $F$ be a $\vartheta$-LHSCB for $K$. Then, for every pair $(x, s) \in \inte(K) \oplus \inte(K^*)$ we have
\begin{equation}
\label{eq:4.2}
\frac{\lambda_\mathrm{max}^{1/2}(F''(x) F_*''(s))}{\delta_F(x,s)} 
\leq 2+\frac{2(\vartheta-1)}{\vartheta(\mu\tilde{\mu}-1)+1} = 2+ \frac{2(\vartheta-1)}{\gamma_\mathrm G (x,s) +1}.
\end{equation}
Therefore, $\widecheck{\xi}_F \leq 2\vartheta.$
Moreover, every pair $(x,s)$ for which $\mu\tilde\mu \geq 2$
satisfies
\begin{equation}
\label{eq:4.3}
\frac{\lambda_\mathrm{max}^{1/2}(F''(x) F_*''(s))}{\delta_F(x,s)} \leq 4.
\end{equation}
\end{thm}

\begin{proof}
Let $K$, $F$ and the pair $(x,s)$ be given as in the assumptions of the statement.
We apply Proposition~\ref{prop:5.2} and deduce
\[
\frac{\lambda_\mathrm{max}^{1/2}(F''(x) F_*''(s))}{\delta_F(x,s)} 
\leq  \frac{2\vartheta\mu\tilde\mu}{1+ \vartheta(\mu\tilde\mu-1)} 
= 2+\frac{2(\vartheta-1)}{\vartheta(\mu\tilde{\mu}-1)+1}.
\]
Since $\mu\tilde{\mu} \geq 1$ for every primal--dual interior pair, the bound $\widecheck{\xi}_F \leq 2\vartheta$ follows.
If $\mu\tilde{\mu} \geq 2$, the right-hand side of~\eqref{eq:4.2} is 
maximized by the smallest value of $\mu\tilde{\mu}$, and~\eqref{eq:4.3} follows.
\end{proof}

\subsection{Local $\widecheck{\xi}_F$ and $\xi_F$} 
Note that the neighbourhood defined by the condition $\mu \tilde{\mu} \leq 2$ 
or, equivalently, $\gamma_\mathrm G(x,s) \leq \vartheta$
(mentioned in Theorem~\ref{thm:5.1}) is much larger than the one defined by 
the Dikin ellipsoid. Moreover, the boundary of the  $\mu \tilde{\mu} \leq 2$
neighborhood approximates the shape of the feasible region much better than 
the Dikin ellipsoid.  For an illustration, see 
Figure~\ref{fig:banded_toeplitz}. 
\begin{figure}
\includegraphics[width=0.4\textwidth]{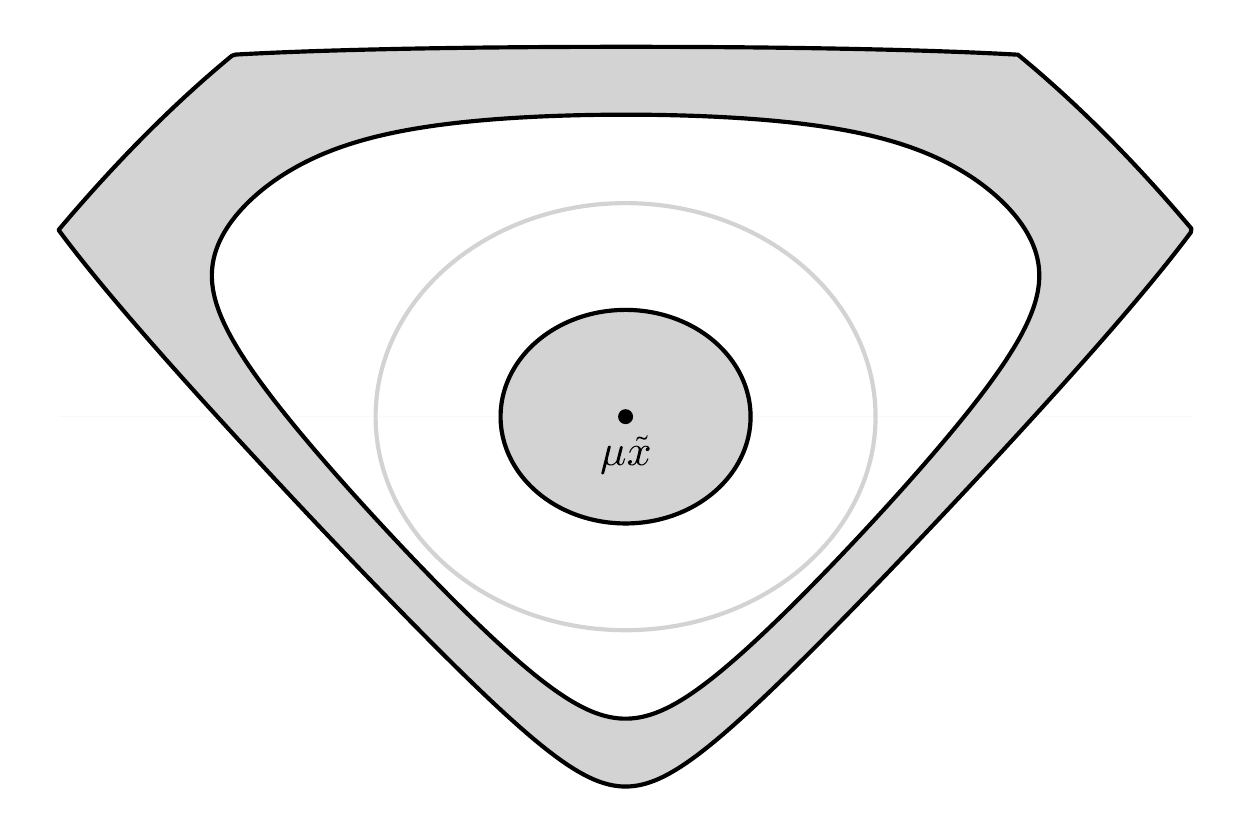}
\caption{A slice of the cone of positive semidefinite tridiagonal Toeplitz 
matrices.
The boundary of the unit Dikin ellipsoid is drawn in gray.
The Dikin ellipsoid with radius $1/2$, where it is known that 
$\widecheck{\xi}_F(x,s) \leq 3$, is shaded.  The region defined by 
$\mu\tilde \mu\geq 2$, where $\widecheck{\xi}_F(x,s) \leq 4$ 
(see Theorem~\ref{thm:5.1}) is also shaded.}
\label{fig:banded_toeplitz}
\end{figure}
These considerations motivate us to look at the
complexity measure $\widecheck{\xi}_F$ locally. We define
\[
\widecheck{\xi}_F(x,s) := 
\frac{\lambda_\mathrm{max}^{1/2}(F''(x) F_*''(s))}{\delta_F(x,s)}.
\]
This is the smallest $\xi$ for which $\widecheck \cT(x,s;\xi)$ is
not empty.
Similarly, we define $\xi_F(x,s)$ as the smallest $\xi$ for which 
$\mathcal T(x,s;\xi)$ is not empty.  
(With these definitions,
$\widecheck \xi_F = \sup_{x,s} {\widecheck \xi_F(x,s)}$ 
and $\xi_F = \sup_{x,s} {\xi_F(x,s)}$.)
In the proof of Theorem~\ref{thm:3.5} we showed that for self-scaled
barriers $F$ and every pair of interior points $x,s$,
the Nesterov--Todd scaling is in $\mathcal T(x,s;\xi)$
for $\xi = \xi_F(x,s)$, so it is optimal for the local measure 
$\xi_F(x,s)$ as well as for the global measure $\xi_F$.
The following list summarizes properties of the local complexity 
measure for $\vartheta$-LHSCBs $F$.
\begin{itemize}
\item
${\xi}_F(x,s) < 1.2115$ for all
$(x,s)$ in a small neighborhood of the central path~\cite[Theorem 6.8]{MT2014}.
\item
${\xi}_F(x,s) \leq 3$ for all $(x,s)$ such that 
$\mu\tilde{x}$ lies in the Dikin ellipsoid with radius $1/2
$ centered at $x$~\cite[Section 8.1]{Obro2019}.
\item
$\widecheck{\xi}_F(x,s) \leq 4$ for all $(x,s)$ such that 
$\mu\tilde{\mu} \geq 2$ (Theorem \ref{thm:5.1}).
\end{itemize}

\subsection{Worst-case primal--dual pairs} 
To gain more insight in the complexity measures $\xi_F$ and $\widecheck \xi_F$,
we now examine the optimization problem in the definition of
$\widecheck\xi_F$.  Using the Rayleigh-Ritz characterization of the 
maximum eigenvalue in \eqref{e-check-xi-eig}, 
we can express $\widecheck{\xi}_F$ in the following variational form:
\begin{eqnarray}
  \label{eq:6.1}
  \widecheck{\xi}_F^{-1} & = &
  \inf_{(x,s)\in\inte(K) \oplus \inte(K^*),v\in \R^n\setminus \{0\}} \sqrt{R(x,s,v)} \cdot \delta_F(x,s),
  \end{eqnarray}
  where
  \[
  R(x, s, v) := \frac{\iprod{F_*''(s)^{-1}v}{v}}{\iprod{F''(x)v}{v}}.
  \]
The worst-case primal--dual pairs $(x,s)$ are pairs that are  
optimal for the minimization problem in~\eqref{eq:6.1}.

In Figures~\ref{fig:nonnegorth} and~\ref{fig:banded_toeplitz2},
and the examples later in the section, it can be noted that the 
worst-case points $\hat x_i$ appear to have additional geometric
properties related to the next proposition. While we proved the
upper bounds $\widecheck{\xi}_F, \xi_F \leq 4/3$ for all negative curvature barriers $F$,
in general, the best bounds we have are $\widecheck{\xi}_F \leq 2\vartheta$
and $\xi_F \leq 9 \vartheta^2$. The best lower bounds we have are given
by optimal self-scaled barriers and for this family $\widecheck{\xi}_F = \xi_F = \rho_{\vartheta} < 4/3$.
At the time of this writing, we do not have examples of $\vartheta$-LHSCBs $F$ with $\widecheck{\xi}_F > \rho_{\vartheta}$,
or with $\xi_F > \rho_{\vartheta}$. 

\begin{prop} \label{p-v-general-F}
Consider the optimization problem
  \begin{equation}
    \label{eq:general_cone_prim}
    \begin{array}{ll}
    \textup{maximize} & \iprod{F''(\tilde x)v}{v}\\
    \textup{subject to} & \iprod{s}{v}= \mu\vartheta\\
    & v \in K,
    \end{array}
  \end{equation}
with variable $v$, where $\mu = \langle s,x\rangle/\vartheta$
and $x\in\inte(K)$ and $s\in\inte(K^*)$ are given. 
Suppose $v$ is feasible in~\eqref{eq:general_cone_prim}
and satisfies 
$\| v - \mu \tilde x \|^2_{\mu \tilde x} = \vartheta(\vartheta - 1)$.
Then $v$ is optimal.  Furthermore, the point
  \begin{equation}
  \label{eq:6.r}
   r:= \tau_{\vartheta} (s- \frac{1}{\mu\vartheta} F''(\tilde x)v)
 \end{equation}
is in $K^*$, and the point
\begin{equation}
  \label{e-hatx-def}
   \hat{x} := (1-\tau_{\vartheta})v + \tau_{\vartheta} \mu \tilde x 
\end{equation}
satisfies $\hat x \in \inte(K)$, $\iprod{s}{\hat x} = \mu\vartheta$, and
$\|\hat x - \mu\tilde x\|_{\mu\tilde x} = \tau_\vartheta/(\tau_\vartheta+1)$.
\end{prop}

\begin{proof}
We first note that the optimal value of~\eqref{eq:general_cone_prim}
is bounded above by $(\vartheta \mu)^2$:  from Lemma~\ref{lem:5.1}, 
\[
\|v\|_{\tilde x} \leq  \langle -F'(\tilde x), v\rangle = 
 \langle s, v\rangle = \vartheta\mu
\]
holds for any feasible $v$.  Hence, for any feasible $v$,
\[
\| v - \mu \tilde x \|^2_{\mu \tilde x} = 
    \frac{1}{\mu^2} \|v\|_{\tilde x}^2  + \vartheta - 2\frac{\iprod{s}{v}}{\mu}
    \leq \vartheta(\vartheta - 1),
\]
and if equality holds, $\|v\|_{\tilde x} = \vartheta\mu$ and $v$ is 
optimal in~\eqref{eq:general_cone_prim}.
The property $r\in K^*$ follows from the optimality conditions 
of~\eqref{eq:general_cone_prim}: $v$ is optimal if it is feasible and
\[
 -2 F''(\tilde x)v + \lambda s \in K^*, \qquad
 \iprod{-2 F''(\tilde x)v + \lambda s}{v} = 0
\]
for some $\lambda$.   From the second condition, $\lambda = 2\mu\vartheta$
if $\|v\|_{\tilde x}^2 = (\vartheta\mu)^2$.
The properties of $\hat x$ follow from
\[
    \norm{\hat{x}-\mu\tilde x}_{\mu\tilde x} = 
 (\tau_{\vartheta}-1) \norm{v-\mu\tilde x}_{\mu\tilde x}
   = (\tau_{\vartheta}-1) \sqrt{\vartheta(\vartheta-1)}
= \frac{\tau_\vartheta}{\tau_\vartheta +1} 
< 1.
\]
This shows that $\hat{x}$ lies in the unit Dikin ellipsoid with 
center $\mu\tilde x$ and therefore $\hat x \in K$.  
Furthermore, $\iprod{s}{\hat{x}} = 
 (1-\tau_{\vartheta})\mu\vartheta + \tau_{\vartheta} \mu 
\vartheta = \mu\vartheta$.
\end{proof}

For self-scaled barriers (Figure~\ref{fig:nonnegorth} and the left-hand
plot of Figure~\ref{fig:banded_toeplitz2}), the worst-case points $\hat x_i$
satisfy~\eqref{e-hatx-def} for points $v$ that are optimal 
in~\eqref{eq:general_cone_prim} with $\|v\|_{\tilde x} = \mu\vartheta$
(equivalently, $\|v-\mu\tilde x\|_{\mu\tilde x}^2 = \vartheta(\vartheta-1)$).
Next, we consider some examples with general LHSCBs.  In the first two
examples, points $v$ with the properties in the 
Proposition~\ref{p-v-general-F} exist, and corresponding points $\hat x$ 
in~\eqref{e-hatx-def} are worst-case points.
However, as Example~\ref{ex-banded-toep} shows, points $v$ with the required 
property do not always exist.

\begin{examp}
The exponential cone is defined as
\[
\Ke := \cl \{ x\in\R^3 : x_1 \geq x_2 \exp(x_3/x_2), \, x_2 > 0 \},
\]
and its dual cone is
\[
\Ke^* = \cl \{ z\in\R^3 : e \cdot z_1 \geq -z_3 \exp(z_2/z_3), 
 \, z_1 > 0, z_3 < 0 \}.
\]
A 3-LHSCB for $\Ke$ is $F(x):=-\ln(\psi(x))$, where 
\[
\psi(x):=x_1 x_2^2 \ln(x_1/x_2) - x_1 x_2 x_3.
\]
 Figure~\ref{fig:expcone_primal}
shows the intersection of $\Ke$ with the hyperplane 
$\{x\in \R^3 : \iprod{s}{x}=1\}$, where $s:=(1,1,-1)$.
\begin{figure}
  \begin{subfigure}{.45\textwidth}
    \centering
    \includegraphics[width=.9\linewidth]{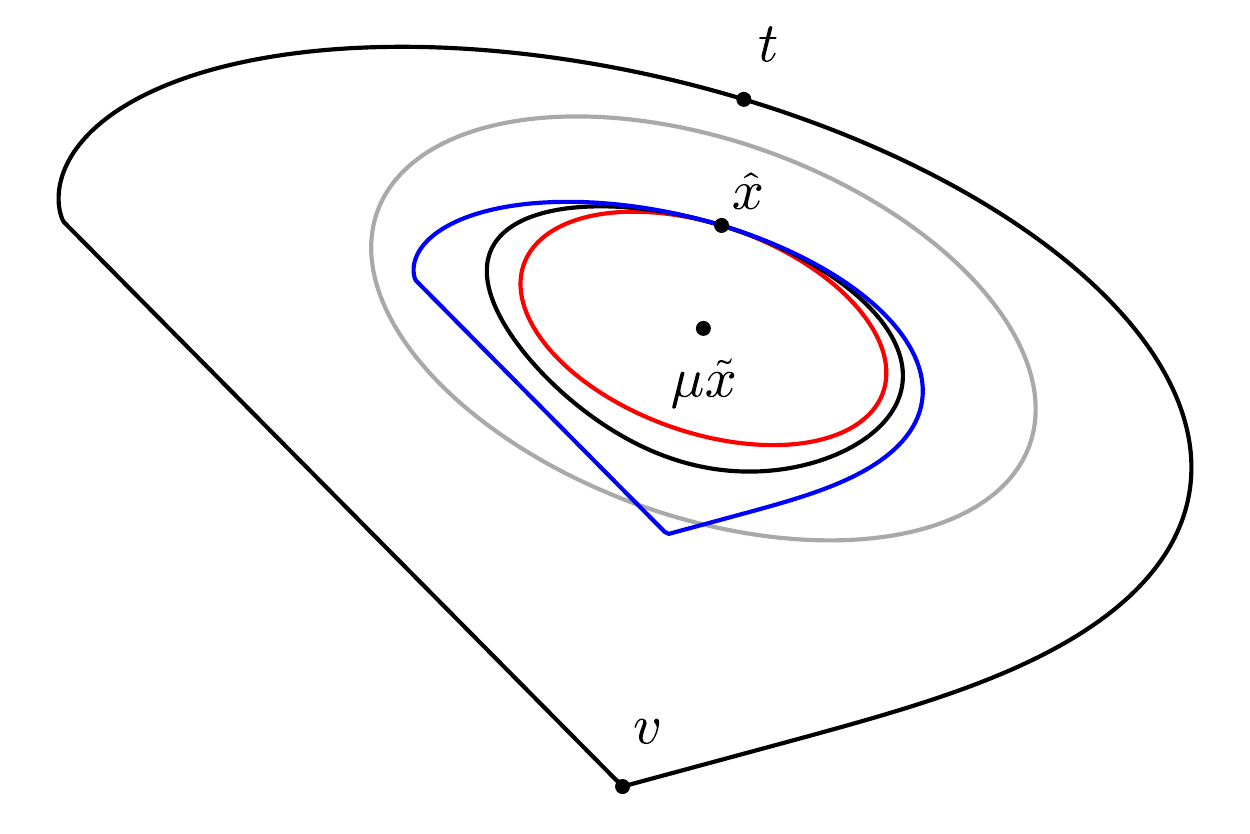}
    \caption{Slice of the primal cone $\Ke$.}
    \label{fig:expcone_primal}
  \end{subfigure}%
  \begin{subfigure}{.45\textwidth}
    \centering
    \includegraphics[width=.9\linewidth]{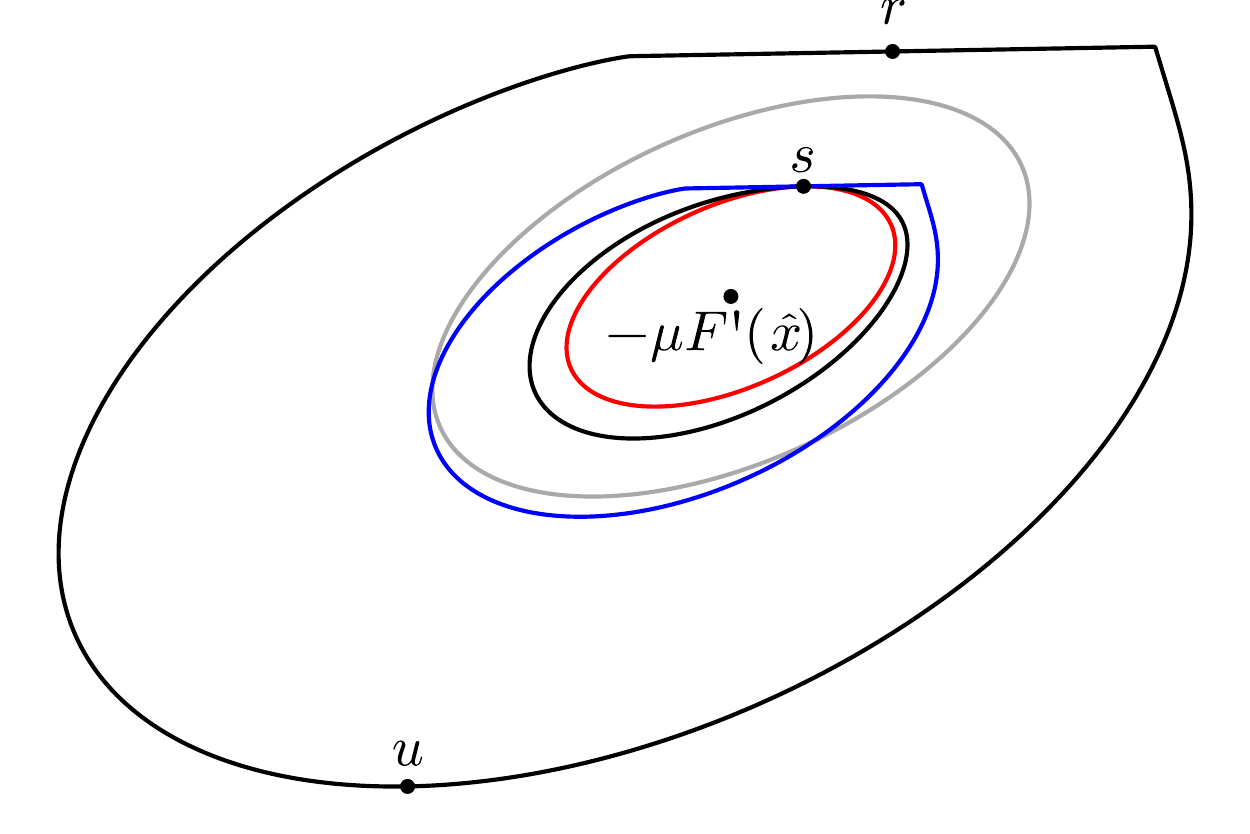}
    \caption{Slice of the dual cone $\Ke^*$.}
    \label{fig:expcone_dual}
  \end{subfigure}  
\caption{Primal and dual exponential cone intersected with the hyperplanes 
normal to $s$ and $\hat{x}$, respectively. 
The inner Dikin ellipsoids (red) 
have radius $\tau_{\vartheta}/(\tau_{\vartheta}+1)$ and include the 
pair of worst-case points $\hat{x}$ and $s$. The boundary of  the region 
$\gamma_{\mathrm G}(x,s)\leq 1/(\tau_{\vartheta}+1)$ is shown in black,
and the boundary of $\gamma_{\infty}(x,s)\leq 1/\tau_{\vartheta}$ in 
blue.  The points $(v,t)$ and $(u,r)$ lie on the boundaries
of $\Ke$ and $\Ke^*$, respectively.}
\label{fig:expcone}
\end{figure}
The extreme point $v:=(0,0,-1)\in \Ke$ satisfies
$\| v - \mu\tilde x \|^2_{\tilde x} = \vartheta(\vartheta-1)$.
The point $\hat x$ defined in~\eqref{e-hatx-def}
lies on the inner Dikin ellipsoid of radius 
$\tau_{\vartheta}/(\tau_{\vartheta}+1)$.
It was verified numerically that $(\hat{x},s,v)$ satisfies 
the stationarity conditions 
for (\ref{eq:6.1}),
and that
$\widecheck{\xi}_F(\hat{x},s)= \rho_\vartheta$.
Furthermore, $r\in \Ke^*$, where $r$ is defined in \eqref{eq:6.r}. 
This point is shown in Figure~\ref{fig:expcone_dual}.

From the primal--dual symmetry in the definition of 
$\widecheck{\xi}(F;\hat{x},s)$, we can also consider worst-case pairs in the 
dual cone $\Ke^*$ intersected with the hyperplane 
$\{y\in \R^3 : \iprod{y}{\hat{x}}=1\}$;
see Figure~\ref{fig:expcone_dual}.
The points
\[
  u := \frac{1}{(1-\tau_{\vartheta})} \left(s - \tau_{\vartheta} \mu F'(\hat{x})\right), \hspace{1cm}
  t := \tau_{\vartheta} \left(\hat{x} - 
 \frac{1} {\mu\vartheta} F''(\hat{x})^{-1}u\right)
\]
satisfy $u \in \Ke^*$ and $t \in \Ke$,  and
$\|u - \mu \tilde s\|_{\mu \tilde s}^* = (\vartheta(\vartheta-1))^{1/2}$,
where 
$\|y\|_{\mu\tilde s}^* := \langle y, F_*''(\mu\tilde s) y\rangle^{1/2}$.
The points $(s,\hat{x},u)$ satisfy the stationarity conditions 
for (\ref{eq:6.1}).

As an extension we also considered a product of exponential cones
$K := K_1 \oplus K_2 \oplus \cdots \oplus K_n$ where each $K_i := \Ke$. 
The barrier $F$ is the summation of the barriers for each $K_i$ 
from the previous example, and has parameter $\vartheta = 3n$. 
Assume $s\in\inte{K^*}$. We define $v := (v_1, v_2, \ldots, v_n) \in K$ where
$v_k := (0,0,\beta)$ and $v_i := (0,0,0)$ for $i\neq k$,
and $\beta$ is chosen such that $\iprod{s}{v}=\mu\vartheta$. Then
$\| v-\mu \tilde x\|_{\mu \tilde x}^2 = \vartheta(\vartheta-1)$.
  Furthermore, the point $\hat x$ defined in~\eqref{e-hatx-def}
  satisfies 
$\widecheck{\xi}_F(\hat{x},s)= \rho_\vartheta$
  and $(\hat{x},s,v)$ satisfies 
  the stationarity assumptions 
  for (\ref{eq:6.1}).
\end{examp}

\begin{examp} \label{ex-toep-wc}
We consider the cone of $3\times 3$ positive semidefinite Toeplitz
matrices with the standard $3$-LHSCB log-det barrier.
Figure~\ref{fig:toeplitz_primal} shows the intersection of the cone 
with the hyperplane defined by $\Tr(sx) = 1$, where $s$ is the identity matrix.
\begin{figure}
\begin{subfigure}{.45\textwidth}
    \includegraphics[width=1.0\textwidth]{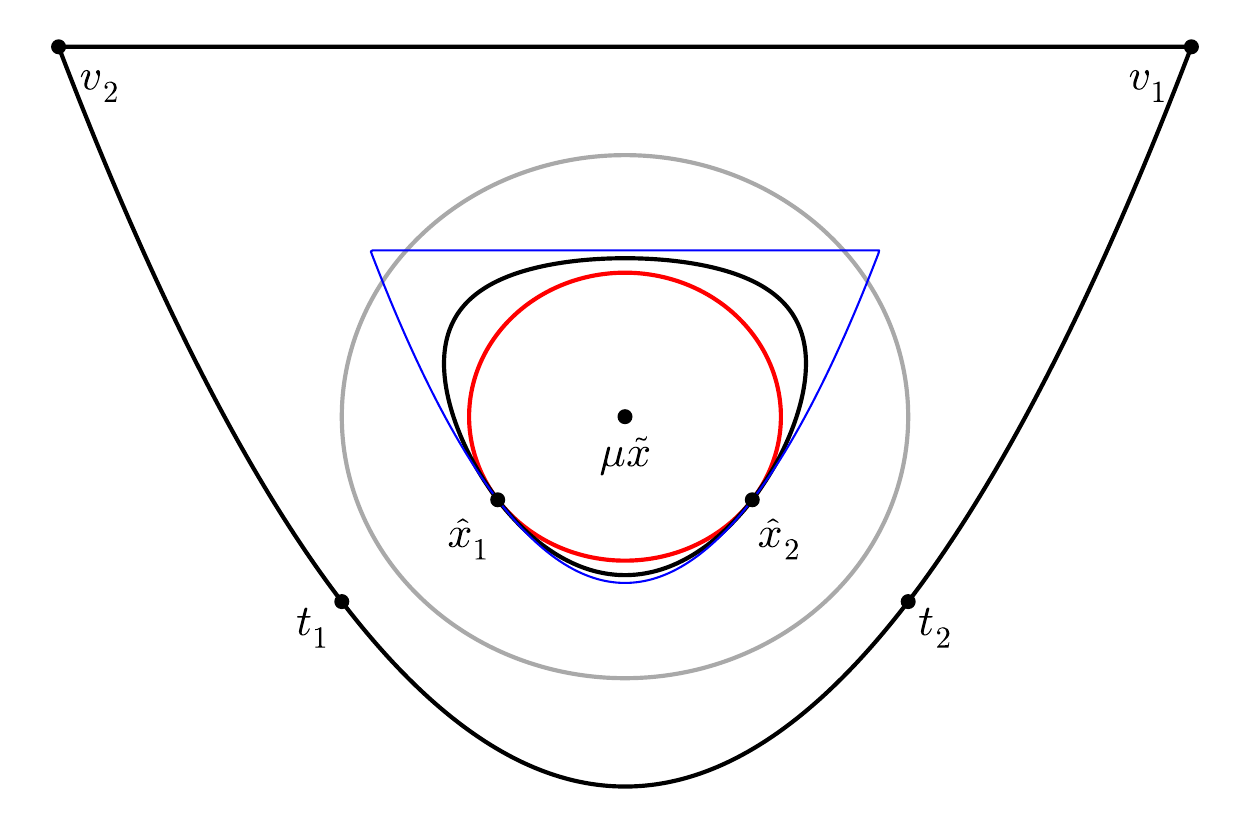}
\caption{Slice of $3\times 3$ PSD Toeplitz matrix cone.} 
\label{fig:toeplitz_primal}
\end{subfigure}
\begin{subfigure}{.45\textwidth}
  \includegraphics[width=1.1\textwidth,trim=2cm 5cm 2cm 3cm,clip]{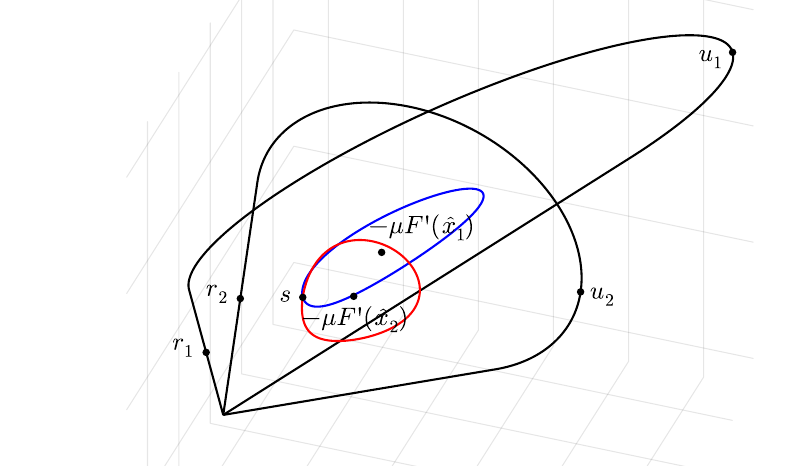}
\caption{Two slices of the dual cone.} 
\label{fig:toeplitz_dual}
\end{subfigure}
\caption{\emph{Left.} Cone of $3\times 3$ positive semidefinite Toeplitz 
matrices cone intersected with the hyperplane defind by $\Tr(sx) = 1$,
where $s$ is the identity matrix.
The unit Dikin ellipsoid is shown in gray. The inner Dikin ellipsoid (red) 
has radius $\tau_{\vartheta}/(\tau_{\vartheta}+1)$ and includes
$\hat{x}_i=(1-\tau_{\vartheta})v_i + \tau_{\vartheta}\mu\tilde x$  
for $i \in \{1,2\}$.  The boundary of  the region
$\gamma_{\mathrm G}(x,s)\leq 1/(\tau_{\vartheta}+1)$ is shown in black,
and the boundary of $\gamma_{\infty}(x,s)\leq 1/\tau_{\vartheta}$ in blue. 
\emph{Right.} Dual cone intersected with the two hyperplanes
  $\{y \in \R^3 : \iprod{y}{\hat{x}_i}=\mu \vartheta\}$ for $i \in \{1,2\}$.
  The smallest $\gamma_{\mathrm G}$ neighborhoods including $s$ and centered at
  $-\mu F'(\hat{x}_1)$ and $-\mu F'(\hat{x}_2)$ are shown in blue and red, 
  respectively.  The two hyperplanes intersect at $s$.
}
\end{figure}
The two extreme points are the matrices are
  \[
  v_1 := \left[\begin{array}{rrr}1 & 1 & 1\\ 1 & 1 & 1\\ 1 & 1 & 1\end{array}\right], \qquad
  v_2 := \left[\begin{array}{rrr}1 & -1 & 1\\ -1 & 1 & -1\\ 1 & -1 & 1\end{array}\right],
  \]
  and satisfy
 $\|v_1 - \mu \tilde x\|_{\mu \tilde x} 
= \|v_2 - \mu \tilde x\|_{\mu \tilde x} = (\vartheta(\vartheta - 1))^{1/2}$.
  The two stationary points on the inner Dikin ellipsoid are given by
  $\hat{x}_i = (1-\tau_{\vartheta})v_i + \tau_{\vartheta} \mu \tilde x$ 
  and satisfy 
$\widecheck{\xi}_F(\hat{x}_i,s) = \rho_\vartheta$
as well as the stationarity assumptions for (\ref{eq:6.1}).

Figure~\ref{fig:toeplitz_dual} shows $K^*$ intersected by the two 
  hyperplanes normal to $\hat{x}_1$ and $\hat{x}_2$, respectively. The points in the figure are defined as
  \[
  u_i := \frac{1}{(1-\tau_{\vartheta})} \left(s - \tau_{\vartheta} \mu F'(\hat{x}_i)\right), \hspace{1cm}
  t_i := \tau_{\vartheta} \left(\hat{x}_i -\frac{1}{\mu\vartheta} 
 F''(\hat{x}_i)^{-1}u\right), \quad i \in \{1,2\},
  \]
  and satisfy the conditions 
  for (\ref{eq:6.1}).
  Points $t_i \in K$ are shown in Figure~\ref{fig:toeplitz_primal}.
The points $r_i$ are defined as
  \[
  r_i:= \tau_{\vartheta} \left( s-
 \frac{1}{\mu \vartheta}F''(\tilde x)v_i\right), \quad i \in \{1,2\}.
  \]
  
\end{examp}

\begin{examp} \label{ex-banded-toep}
We consider the cone of $5\times 5$ positive semidefinite tridiagonal 
Toeplitz matrices, with the standard $5$-LHSCB log-det barrier.
Figure~\ref{fig:banded_toeplitz2} shows
the intersection of the cone with a hyperplane normal to the identity matrix. 
For this example, all points $v\in K$ on the hyperplane have 
$\|v-\mu\tilde x\|_{\mu \tilde x}^2 < \vartheta(\vartheta - 1)$, 
and there are no points $\hat{x}\in K$ on the hyperplane satisfying 
 $\widecheck{\xi}_F(\hat{x},s) = \rho_\vartheta$.
\end{examp}

\end{section}

\begin{section}{Conclusion}

\label{sec:conclusion}

We showed that $\xi_F \leq 4/3$ for all LHSCBs with negative curvature
and that the integral scaling is 
an element of $\cT(x,s;4/3)$.
An immediate consequence of this result 
is that the primal--dual potential reduction algorithm 
from~\cite{Tuncel2001} 
when implemented using the integral primal--dual scalings~\cite{MT2014} has
iteration complexity $O\left(\vartheta^{1/2}\ln(1/\epsilon)\right)$.
This result positively impacts other primal--dual algorithms, e.g., 
the predictor--corrector algorithm analyzed in~\cite{MT2014}. 
For negative curvature barriers, our results imply that in primal--dual
predictor-corrector algorithms, implemented with the integral scaling (or its suitable
approximations),
we may take significantly larger predictor steps and
still maintain $O\left(\vartheta^{1/2}\ln(1/\epsilon)\right)$ iteration 
complexity bounds.

Our result on negative curvature barriers
covers a large class of convex cones including hyperbolicity
cones arising from hyperbolic polynomials. However, the following problems
remain open:

\begin{enumerate}
\item
For every LHSCB $F$, characterize the pairs $(x,s)$ attaining the
value $\widecheck{\xi}_F$.
\item
For every LHSCB $F$, characterize the pairs $(x,s)$ attaining the
value $\xi_F$.
\item
Does every convex cone admit
an LHSCB with negative curvature?
\item
Does every convex cone admit an LHSCB $F$ with $\xi_F \leq 4/3$?
\end{enumerate}
  
  Question (3) was originally posed in \cite{NT2016}.
  If the answer to question (3) is ``yes,'' then our results imply
  that the answer to question (4) is also ``yes.''

  Another line of related future research concerns the investigation of characterizations of
  worst-case pairs $(x,s)$ (i.e., those which satisfy $\widecheck{\xi}_F(x,s) = \xi_F$) with respect
  to the properties of the barrier $F$ and the geometry of $K$ as suggested by 
many of the figures in the paper.
  
 \end{section}

\bibliographystyle{plain}
\bibliography{DahlTuncelVandenberghe2021}
\end{document}